\documentclass[english]{smfart}
\usepackage{mathrsfs}
\usepackage{smfthm}
\usepackage{amssymb}
\usepackage{latexsym}
\usepackage{verbatim}
\usepackage{amsmath}
\usepackage{graphicx,color}
\usepackage[all]{xy}
\usepackage{pstricks}
\usepackage{pst-plot}
\usepackage{epsfig}
\usepackage{mathrsfs}
 \theoremstyle{plain}
 \newtheorem{thm}{Theorem}[section]
 \newtheorem{cor}[thm]{Corollary}
 \newtheorem{lem}[thm]{Lemma}
 
 \newtheorem{propo}[thm]{Proposition}
 \theoremstyle {definition}
 
 \theoremstyle{remark}
 \newtheorem*{rem}{Remark}

\def\polhk#1{\setbox0=\hbox{#1}{\ooalign{\hidewidth
    \lower1.5ex\hbox{`}\hidewidth\crcr\unhbox0}}}
\renewcommand{\Im}{\operatorname{Im}} 
\renewcommand{\Re}{\operatorname{Re}}
\newcommand{\Dom}{\operatorname{Dom}\,}
\newcommand{\diam}{\operatorname{diam}\,}
\newcommand{\area}{\operatorname{area}}
\newcommand{\daron}{\operatorname{int}\,}
\newcommand{\ex}{\operatorname{\mathbb{E}xp}}
\newcommand{\ecc}{\operatorname{Ecc}} 
\newcommand{\eps}{\varepsilon}

\newcommand{\irr}{Irrat_{\geq N}}
\newcommand{\pc}{\mathcal{PC}}
\newcommand{\alf}{\alpha}
\newcommand{\ff}{\mathcal{I}\hspace{-1pt}\mathcal{S}}
\newcommand{\fff}{\ff_{(0,\alf^*]}}
\newcommand{\orb}{\mathcal{O}}
\newcommand{\RS}{\mathbb{\hat{C}}}
\newcommand{\vfi}{\varphi}
\newcommand{\cc}{\mathbb{C}}
\newcommand{\ra}{\rightarrow}
\newcommand{\ea}{e^{2\pi \alf \mathbf{i}}}
\newcommand{\p}{\mathcal{P}}
\newcommand{\C}{\mathcal{C}}
\newcommand{\Csh}{\mathcal{C}^\sharp}
\newcommand{\rr}{\mathcal{R}}
\newcommand{\D}{\mathcal{D}}
\newcommand{\K}{\mathcal{K}}
\newcommand{\G}{\mathcal{G}}

\newcommand{\cp}{\textup{cp}}
\newcommand{\cv}{\textup{cv}}
\newcommand{\Bi}{\textbf{i}}
\newcommand{\Bk}{\boldsymbol{k}}
\newcommand{\co}[1]{^{\circ {#1}}}
\begin{document}
\title[Typical orbits of complex quadratic polynomials]{Typical orbits of quadratic polynomials
with a neutral fixed point: Brjuno type}
\author{Davoud Cheraghi}
\address{Mathematics Institute, University of Warwick,   Coventry CV4-7AL, UK}
\email{d.cheraghi@warwick.ac.uk}
\keywords{typical orbits, quadratic polynomials, neutral fixed point, attractor}
\subjclass{Primary 37F50; Secondary 58F12}
\date{\today}

\begin{abstract}
We describe the topological behavior of typical \textit{orbits} of complex quadratic polynomials 
$P_\alf(z)=\ea z+ z^2$, with $\alf$ of \textit{high return} type.
Here we prove that for such Brjuno values of $\alf$ the closure of the \textit{critical  orbit}, 
which is the \textit{measure theoretic attractor} of the map, has zero area. 
Then we show that the \textit{limit set} of the orbit of a typical point in 
the \textit{Julia set} of $P_\alf$ is equal to the closure of the critical orbit.
Our method is based on the near parabolic renormalization of Inou-Shishikura, and a uniform 
optimal estimate on the derivative of the Fatou coordinate that we prove here. 
\end{abstract}
\maketitle
\renewcommand{\thethm}{\Alph{thm}}
\section*{Introduction}
The local, semi-local, and global dynamics of the maps 
\[P_\alf(z):= \ea z+z^2:\cc\ra\cc,\]
for irrational values of $\alf$, have been extensively studied through various methods over the last decades.
The aim of this work is to describe the topological behavior of the orbit of typical points under these maps. 
This is a step toward understanding the measurable dynamics of these maps.   

The \textit{post-critical} set of $P_\alf$ is defined as the closure of the orbit of the critical value;
\[\pc(P_\alf):=\overline{\cup_{j=1}^\infty P_\alf\co{j}(-\ea/2)}.\]
It is well-known \cite{Ly83b} that $\pc(P_\alf)$ is the \textit{measure theoretic attractor} of the dynamics of 
$P_\alf$ on its \textit{Julia set} $J(P_\alf)$. 
That is, the orbit of Lebesgue almost every point in $J(P_\alf)$ eventually stays in any given 
neighborhood of $\pc(P_\alf)$. 
To understand the long term behavior of typical orbits in $J(P_\alf)$, one needs to understand the structure   
of the set $\pc(P_\alf)$ and the iterates of $P_\alf$ near $\pc(P_\alf)$.  

Two different scenarios occur depending on the local dynamics of $P_\alf$ at zero.
Let $\alf:=[0; a_1,a_2,a_3,\dots]$ denote the continued fraction expansion of $\alf$ with the 
convergents $p_n/q_n:=[0;a_1,a_2, \dots, a_n]$. 
By classical theorems of Siegel and Brjuno~\cite{Sie42,brj71}, if the series 
$\sum_{n=1}^{\infty} \log q_{n+1}/q_n$ is finite, the map $P_\alf$ is \textit{linearizable} at zero, 
i.e.\ it is locally conformally conjugate to a rotation. 
When linearizable, the maximal domain of linearization is called the \textit{Siegel disk} of $P_\alf$. 
The values of $\alf$ for which the above sum is convergent are called \textit{Brjuno} numbers.  
On the other hand, Yoccoz \cite{Yoc95} proved that the convergence of this series is necessary for the  
linearizability of $P_\alf$. 
In \cite{Ma83}, Ma{\~n}{\'e} proves that for irrational $\alf$ the orbit of the critical point of $P_\alf$ is recurrent, 
and accumulates at the boundary of the Siegel disk (at zero) when $P_\alf$ is linearizable 
(non-linearizable, respectively). 

Petersen and Zakeri in \cite{PZ04} describ the topology and geometry of the dynamics of the 
linearizable maps $P_\alf $ for a.e.\ $\alf\in[0,1]$. 
For these values, they show that $J(P_\alf)$ has zero area.
However, a rather surprising result of Buff and Ch\'eritat \cite{BC06} states that there are parameters 
$\alf$, both of Brjuno and non-Brjuno type, for which $J(P_\alf)$ has positive area.
It is conjectured~\cite{Cher09} that for generic values of $\alf$, $J(P_\alf)$ has positive area.
Knowing all these, still there is not a single example of a quadratic map with a non-linearizable fixed point 
whose local dynamics is completely understood.

A major breakthrough in the field by Inou and Shishikura \cite{IS06} has allowed further progress in the study 
of these maps.
It is an essential part of \cite{BC06} and is used in \cite{Sh10} to show that the 
boundary of these Siegel disks are Jordan curves.  
Roughly speaking, Inou-Shishikura show that successive \textit{renormalizations} of $P_\alf$ 
(a sophisticated version of successive return maps depicted in Figure~\ref{easy-renormalization}) 
are defined on ``large enough'' domains, and belong to a compact class of maps. 
This scheme requires the digits in the expansion of $\alf$ to be larger than some constant $N$,\!
\footnote{However, they conjecture that $N=1$.}, 
i.e.\ 
\[\alf\in \irr:=\{[0;a_1,a_2,\dots]\in (0,1)\mid \inf_i a_i \geq N\}.\] 

In \cite{Ch10-I} we started a systematic study of the measurable dynamics of these maps by quantifying the 
renormalization scheme of Inou-Shishikura as well as estimating the changes of coordinates between consecutive  
renormalization levels. 
This was mainly applied to the study of non-linearizable maps. 
In particular, we showed that for non-Brjuno values of $\alf$, $\pc(P_\alf)$ is non-uniformly porous 
\footnote{A set $E\subseteq \cc$ is said to be non-uniformly porous, if there exists a 
$\lambda\in (0,1)$ satisfying the following property. 
For every $z\in E$ there exists a sequence of real numbers $r_n\to 0$ such that the ball of radius $r_n$ 
about $z$ contains a ball of radius $\lambda r_n$ disjoint from $E$, for every $n$.}
(and hence has zero area). 
Here, we prove the following counterpart.
\begin{thm} \label{pc-area}
For every Brjuno $\alf\in\irr$,  $\pc(P_\alf)$ has zero area.
\end{thm}

When $\alf$ is non-Brjuno, the arithmetic of $\alf$ satisfies a particular property that we exploited in 
\cite{Ch10-I} to prove the existence of complementary holes at arbitrarily small scales on $\pc(P_\alf)\setminus \{0\}$. 
But that arithmetic property does not necessarily hold for all Brjuno values of $\alf$, and hence,
we need a different approach to prove the above theorem.
In particular, the global estimates that appeared in \cite{Ch10-I} I are not enough and 
we need to prove a uniform infinitesimal estimate on the changes of coordinates.

Note that by Mane's Theorem, the boundary of the Siegel disk is contained in the post-critical set, and hence, 
must have zero area by the above result.
The next statement is an immediate corollary of Theorem~\ref{pc-area}. 
\begin{cor}\label{non-recurrent}
For every Brjuno $\alf\in\irr$, almost every point in $J(P_\alf)$ is non-recurrent. 
In particular, there is no finite absolutely continuous invariant measure on $J(P_\alf)$.  
\end{cor}
Theorem~\ref{pc-area} and its counterpart for non-Brjuno values of $\alf$ in \cite{Ch10-I} 
enable us to prove the following statement here. 
\begin{thm}\label{acc-on-cp}
For all $\alf\in\irr$, the limit set of the orbit of almost every point in $J(P_\alf)$ is equal to $\pc (P_\alf)$.
\end{thm}
Let $\overline{\Delta}_\alf$ denote the closure of the Siegel disk of $P_\alf$.
By \cite{Her85} there are Brjuno values of $\alf$ for which $\overline{\Delta}_\alf$ does not contain the 
critical point, and therefore, $\pc(P_\alf)\setminus\overline{\Delta}_\alf$ is non-empty. 
Conjecturally, this set is homeomorphic to the Cantor bouquet minus its root. 
Our analysis of the post-critical set allows us to prove the following geometric property of these decorations.
\begin{thm}\label{thm:porosity-of-decorations}
For all $\alf\in\irr$, $\pc(P_\alf)$ is non-uniformly porous at every point in the set
$\pc(P_\alf)\setminus\overline{\Delta}_\alf$.  
\footnote{It follows from the proof of this theorem that under various arithmetic conditions on $\alf$, one can replace $\pc(P_\alf)\setminus\overline{\Delta}_\alf$ by $\pc(P_\alf)$. 
However, we do not believe that one can do this for all Brjuno $\alf$.} 
\end{thm}
\begin{figure}
\begin{center}
\begin{pspicture}(-.5,0)(5,4.1)
\psccurve[linewidth=.3pt,showpoints=false,fillstyle=solid,fillcolor=lightgray](1.8,1)(1.7,1.3)(1,1.5)(1.2,1.7)(.3,1.4)(0,1)(.6,.2)(1.4,.4)
\psdots[dotsize=1pt](1,3)(1,1)
\rput(.85,3.15){\small{$0$}}
\rput(.85,1.15){\small{$0$}}
\rput(.2,3.7){\small{$P_\alf$}}
\psccurve[linewidth=.3pt](2,3)(1,4)(.2,3)(1,2)
\pspolygon[fillstyle=solid,fillcolor=lightgray,linecolor=white](1,3)(1.6,2.4)(1.75,2.6)
\psline[linewidth=.6pt](1,3)(1.6,2.4)
\pscurve[linewidth=.6pt](1,3)(1.5,2.7)(1.75,2.6)
\psline[linewidth=.6pt](1.6,2.4)(1.75,2.6)
\psline[linecolor=gray,linewidth=.3pt]{->}(1.5,2.6)(1.7,2.9)
\psline[linecolor=gray,linewidth=.3pt]{->}(1.7,3)(1.5,3.3)
\psline[linecolor=gray,linewidth=.3pt]{->}(1.4,3.4)(1,3.5)
\psline[linecolor=gray,linewidth=.3pt]{->}(.9,3.42)(.6,3.2)
\psline[linecolor=gray,linewidth=.3pt]{->}(.6,3.1)(.7,2.7)
\psline[linecolor=gray,linewidth=.3pt]{->}(.76,2.64)(1.1,2.6)
\psline[linecolor=gray,linewidth=.3pt]{->}(1.12,2.6)(1.4,2.7)

\psline[linewidth=.3pt]{->}(1.7,2.7)(2.8,3)

\psellipse[linewidth=.3pt](3.4,3.8)(.3,.1)
\psellipse[linewidth=.3pt](3.4,2.4)(.3,.1)
\psline[linewidth=.3pt](3.1,3.8)(3.1,2.4)
\psline[linewidth=.3pt](3.7,3.8)(3.7,2.4)

\pscustom[linewidth=.3pt,fillstyle=solid,fillcolor=lightgray]{
\psline[linewidth=.3pt](3.1,3.8)(3.1,3.2)
\pscurve[linewidth=.3pt](3.1,3.2)(3.35,3.3)(3.25,3.1)(3.7,3.2)
\psline[linewidth=.3pt](3.7,3.2)(3.7,3.8)
\pscurve[linewidth=.3pt](3.7,3.8)(3.4,3.7)(3.1,3.8)}

\pscurve[linewidth=.3pt,linecolor=gray, linestyle=dashed](3.1,3.2)(3.4,3.25)(3.7,3.2)

\pscurve[linewidth=.3pt]{->}(3.8,3.4)(4.2,3.5)(3.8,3.6)
\rput(4.5,3.5){\small{$E_{P_\alf}$}}
\rput(3.9,2.2){\small{$\cc/\mathbb{Z}$}}

\psline[linewidth=.3pt]{->}(2.8,2.6)(2,1.6)
\rput(2.4,2){\small{$e^{2\pi \Bi z}$}}
\rput(0,1.55){\small{$\rr P_\alf$}}
\psframe[linewidth=.2pt](-.5,0)(5,4.1)
\end{pspicture}
\caption{Identifying the sides of a sector landing at zero under $P_\alf$ one obtains a half infinite cylinder 
which projects onto a neighborhood of zero under $e^{2\pi \Bi z}$. 
The return map to this sector under $P_\alf$ induces a map on that neighborhood, $\rr P_\alf$, 
called the renormalization of $P_\alf$.} 
\label{easy-renormalization}
\end{center}
\end{figure}
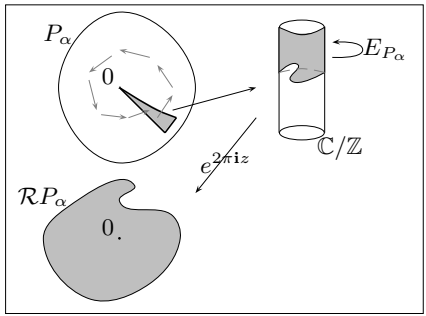

Given an analytic germ $f$ defined near zero with $f(0)=0$ and $|f'(0)|=1$, one blows up a small neighborhood 
of $0$ under a suitable covering transformation and obtains a new map $F$ defined near infinity that is close 
to the translation by one. 
The \textit{Fatou coordinate} of $F$ is a conformal change of coordinate that conjugates $F$ to the translation by one 
above a horizontal line.
Estimates on these changes of coordinates have been the core part of most of the results concerning 
the dynamics of near parabolic maps, see \cite{Sh98, ABC04, IS06, BC06, Ch10-I}, et al. 
However, the errors of the infinitesimal estimates proved so far, mainly in \cite{Sh00, Sh98, IS06}, 
have divergent integrals over infinite vertical strips of width one. 
In Section~\ref{S:perturbed-coordinate} we present a new estimate on the derivative of the Fatou 
coordinate with exponentially decaying (as a function of Imaginary part) error.
Indeed, we show that given a compact class of germs $f$ (in the compact open topology), the $L^1$ norm of 
the error (over infinite vertical strips) is uniformly bounded by a constant divided by the rotation of $f$ at $0$.
This estimate allows us to prove the results under the sharp Brjuno condition on $\alf$ here, and will be 
used to describe the statistical properties of the quadratic maps $P_\alf$ of high type in \cite{AC11}.
Moreover, one may use these estimates to answer some major optimality questions, like Herman's condition, 
in the class of quadratic maps of high type. 
Our approach to prove this estimate is different from the earlier ones, and is based on proposing an 
explicit $C^2$ change of coordinate that satisfies the optimal estimate and is
``nearly conformal'' and ``nearly harmonic''. 
It is an application of Green's Integral Formula (Hilbert transform) to compare the actual perturbed Fatou 
coordinate to this model.

For the parameters at the other end, i.e.\ when $\sup_i a_i<\infty$, the dynamic of $P_\alf$ has been 
beautifully described in \cite{Mc04}.
See \cite{BC04,BC07,Ya08,Blo10} and the references therein for some recent advancements in other 
aspects of the dynamics of these quadratic maps.
\subsection*{Frequently used notations}
\begin{itemize}
\setlength{\leftmargin}{-1em}
\item[--]$:=$ is used when a notation appears for the first time. 
\item[--]$\mathbb{Z}$, $\mathbb{Q}$, $\mathbb{R}$, $\mathbb{C}$ denote the integer, rational, real, and complex numbers, respectively.
\item[--]The bold face $\Bi$ denotes the imaginary unit complex number.
\item[--]$\Re z$, $\Im z$, and $|z|$ denote the real part, the imaginary part, and the absolute value of a complex number
$z$, respectively.
\item[--]$B(y,\delta)\subset \cc$ denotes the ball of radius $\delta$ around $y$ in the Euclidean metric.
Given a set $X\subseteq \cc$, $B(X, \delta):=\cup_{y\in X}B(y, \delta)$.
\item[--] $\daron(S)$ denotes the interior of a set $S\subset \cc$.
\item[--]$f\co{n}$ denotes the $n$ times composition of a map $f$ with itself, $f\co{0}=id$.
\item[--]$\Dom f$, $J(f)$, and $\pc(f)$ denote the domain of definition, the Julia set, and the post-critical 
set of a map $f$, respectively.
\item[--]Univalent map refers to a one to one holomorphic map.
\item[--]Given $g\colon\!\!\Dom g\ra\cc$, with only one critical point in its domain of definition, $\cp_g$ 
and $\cv_g$ denote the critical point and the critical value of $g$, respectively. 
\item[--] For $x\in\mathbb{R}$, $\lfloor x \rfloor$ denotes the largest integer less than or equal to $x$.
\end{itemize}   
\renewcommand{\thethm}{\thesection.\arabic{thm}}
\section{Preliminaries on renormalization}\label{sec:prelim} 
\subsection{Inou-Shishikura class}
Consider the cubic polynomial $P(z):=z(1+z)^2$. 
This polynomial has a \textit{parabolic} fixed point at $0$, that is, a fixed point of multiplier $e^{2\pi\alf\Bi}$
with $\alf\in \mathbb{Q}$. 
It has a critical point at $\cp_P:=-1/3$ with $P(\cp_P):=\cv_P=-4/27$, and another 
critical point at $-1$ which is mapped to $0$ under $P$.

Consider the ellipse 
\[E:= \Big\{x+\Bi y\in \cc\mid (\frac{x+0.18}{1.24})^2+(\frac{y}{1.04})^2\leq 1 \Big\},\]
and let 
\begin{equation}\label{U}
U:= g(\RS \setminus E), \text{ where } g(z):=-\frac{4z}{(1+z)^2}.
\end{equation}
The domain $U$ contains $0$ and $\cp_P$, but not $-1$.
Following \cite{IS06}, we define the classes of maps
\begin{displaymath}
\ff\!:=\Big\{f:=P\circ \vfi^{-1}\!\!:U_f \rightarrow \cc \;\Big|%
\begin{array}{l} 
\text{$\vfi\colon U \ra U_f$ is univalent, $\vfi(0)=0$, $\vfi'(0)=1$,}\\ 
\text{and $\vfi$ has a quasi-conformal\footnotemark{} 
extension to $\cc$.} 
\end{array}
\Big\},
\end{displaymath}
and, 
\[\ff_A:=\{f(\ea z) \mid f\in \ff, \textrm{ and } \alf \in A\}, \text{ where } A\subseteq \mathbb{R}.\]\footnotetext{See \cite{Ah66} for the definition of quasi-conformal mappings.}
Abusing the notation, $\ff_\beta$ denotes the set $\ff_{\{\beta\}}$, for $\beta\in \mathbb{R}$.

Every map in $\ff$ has a parabolic fixed point at $0$ and a unique critical point at 
$\cp_f:=\vfi(-1/3)\in U_f$. 

Consider a map $h\colon \!\!\Dom h \ra \cc$, where $\Dom h$ denotes the domain of definition (always assumed to be open) of $h$. 
Given a compact set $K\subset \Dom h$ and an $\eps>0$, a neighborhood of $h$ in the compact-open topology 
is defined as 
\[\mathcal{N}(h;K,\eps):=\{g\colon\!\! \Dom g\ra \cc \mid K\subset \Dom g , \textrm{ and } \sup_{z\in K} 
|g(z)-h(z)|<\eps\}.\]
In this topology, a sequence $h_n:\Dom h_n\ra \cc$, $n=1,2,\dots$, \textit{converges} to a map $h$ if for any neighborhood of $h$ defined as above, $h_n$ is contained in that neighborhood for sufficiently large $n$. 
Note that the maps $h_n$ need not be defined on the same domains.

The class $\ff_A$ naturally embeds into the space of univalent maps on the unit disk with a neutral fixed 
point at $0$. 
Therefore, it is a precompact class in the compact-open topology.
In particular, it follows from the K\"{o}ebe distortion Theorem that $\{h''(0)\mid h\in \ff_{\mathbb{R}}\}$ is 
relatively compact in $\mathbb{C}\setminus \{0\}$. 

Any map $h=\ea f\in\ff_{\alf}$ has a fixed point at $0$ with multiplier $\ea$. 
Moreover, if $\alf$ is small, $h$ has another fixed point $\sigma_h\neq 0$ near $0$ in $U_h$. 
The $\sigma_h$ fixed point depends continuously on $h$ and has asymptotic expansion 
$\sigma_h=-4\pi \alf \Bi/f''(0)+o(\alf)$, when $h$ converges to $f\in\ff$ in a fixed neighborhood of $0$. 
Clearly $\sigma_h \ra 0$ as $\alf \ra 0$.  

The following theorem introduces a useful coordinate to study the local dynamics of maps in $\ff_{\alf}$. 
See Figure~\ref{petal} for a geometric description of the following Theorem.
\begin{thm}[Inou--Shishikura \cite{IS06}]\label{Ino-Shi1} 
There exists $\alf_*\in (0,1)$ such that for every $h\colon U_h \ra \cc$ in $\ff_\alf$ (or $h=P_\alf:\cc\ra\cc$) 
with $\alf \in (0,\alf_*]$, there exist a domain $\p_h \subset U_h$ and a univalent map $\Phi_h\colon \p_h \ra \cc$
satisfying the following properties:
\begin{itemize}  
\item[i.] The domain $\p_h$ is bounded by piecewise smooth curves and is compactly contained in $U_h$. 
Moreover, it contains $\cp_h$, $0$, and $\sigma_h$ on its boundary.
\item[ii.] There exists a continuous branch of argument defined on $\p_h$ such that 
\[\max_{w,w'\in \p_h} |\arg(w)-\arg(w')|\leq 2 \pi \hat{\Bk}.\]
\item[iii.] $\Phi_h(\p_h)=\{w \in \cc \mid 0 < \Re(w) < \lfloor 1/\alf \rfloor-\Bk\}$, 
$\Im \Phi_h(z) \ra +\infty$ as $z\in \p_h\ra 0$, and $\Im \Phi_h(z)\ra-\infty$ as $z \in \p_h \ra \sigma_h$.
\item[iv.] $\Phi_h$ satisfies the Abel functional equation on $\p_h$, that is, 
\[\Phi_h(h(z))=\Phi_h(z)+1, \text{ whenever $z$ and $h(z)$ belong to $\p_h$}.\] 
Furthermore, $\Phi_h$ is unique once normalized by $\Phi_h(\cp_h)=0$.
\item[v.] The normalized map $\Phi_h$ depends continuously on $h$.     
\end{itemize}
\end{thm}

\begin{figure}[ht]
\begin{center}
  \begin{pspicture}(10.4,5.3)
\epsfxsize=5.6cm 
\rput(2.85,2.81){\epsfbox{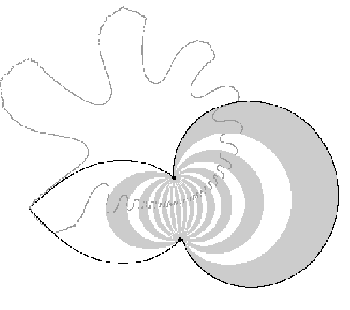}}
  \psset{xunit=.7cm}
  \psset{yunit=.7cm}
     \pspolygon*[linecolor=lightgray,fillcolor=lightgray](10.5,.5)(10.5,6.5)(11,6.5)(11,.5)
     \pspolygon*[linecolor=lightgray,fillcolor=lightgray](9.5,.5)(9.5,6.5)(10,6.5)(10,.5)
     \pspolygon*[linecolor=lightgray,fillcolor=lightgray](13,.5)(13,6.5)(13.5,6.5)(13.5,.5)
     \pspolygon*[linecolor=lightgray,fillcolor=lightgray](12,.5)(12,6.5)(12.5,6.5)(12.5,.5) 
     \psline(13.5,3.4)(13.5,3.6)
     \psline(13,3.4)(13,3.6)
     \psline(12.5,3.4)(12.5,3.6)
     \psline(12,3.4)(12,3.6)

     \psline(9.5,3.4)(9.5,3.6)
     \psline(10,3.4)(10,3.6)
     \psline(10.5,3.4)(10.5,3.6)
     \psline(11,3.4)(11,3.6)
    
     \rput(8.8,3.1){$0$}
     \rput(9.5,3.1){$1$}
     \rput(10,3.1){$2$}
     \rput(10.5,3.1){$3$}
     \rput(11,3.1){$4$}

     \rput(11.5,4){$\cdots$}
     \psaxes*[linewidth=1.2pt,labels=none,ticks=none]{->}(9,3.5)(8.5,.5)(14,6.5)
     \pscurve{->}(6.8,4.3)(8.5,5)(10,4.5)
     \rput(8.2,5.4){$\Phi_h$}
     \rput(4.6,3.6){$0$}
     \rput(4.8,1.9){$\sigma_h$}
     \psdots(.7,2.8)
     \rput(.1,3){$\text{cp}_h$}
     \rput(14,3){$\frac{1}{\alpha}-\Bk$}
     \rput(6.5,2.2){$\mathcal{P}_h$}
     \rput(1,7){$U_h$}
   \end{pspicture}
\caption{A perturbed Fatou coordinate $\Phi_h$ and its domain of definition $\mathcal{P}_h$. 
The first few iterates of $\cp_h$ under $h$ appears in the Figure as the boundary of the amoeba containing $0$.}
\label{petal}
\end{center}
\end{figure}
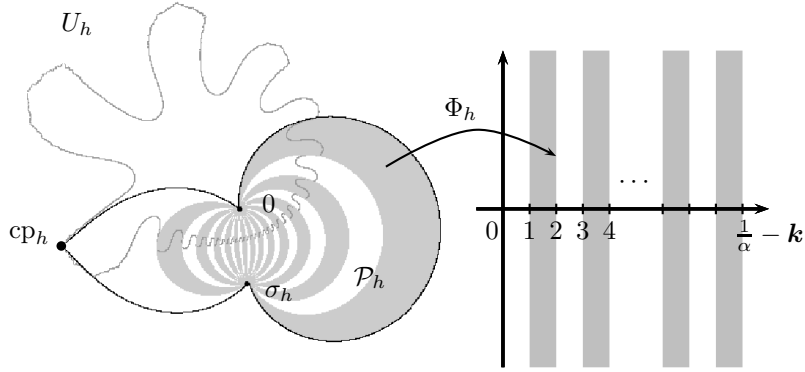
The map $\Phi_h: \mathcal{P}_h\to \cc$ obtained in the above theorem is called the \textit{perturbed Fatou coordinate}, 
or the \textit{Fatou coordinate} for short,  of $h$.

The class $\ff$ is denoted by $\mathcal{F}_1$ in \cite{IS06}. 
All parts in the above theorem, except the existence of uniform $\hat{\Bk}$ and $\Bk$ in ii. and iii., follow readily 
from Theorem 2.1, Main Theorems 1, 3, and Corollary 4.2 in \cite{IS06}. 
Parts ii. and iii. also follow from those results but require some extra work. 
A detailed treatment of these statements are given in \cite[Proposition 12]{BC06}.        

\subsection{Renormalization}
Let $h\colon U_h \ra \cc$ either be in $\ff_{\alf}$ or be the quadratic polynomial $P_\alf$, with $\alf$ in $(0,\alf_*]$. 
Let $\Phi_h\colon \p_h \ra \cc$ denote the normalized Fatou coordinate of $h$.
Define 
\begin{equation}\label{sector-def}
\begin{aligned}
&\C_h:=\{z\in \p_h : 1/2 \leq \Re(\Phi_h(z)) \leq 3/2 \: ,\: -2< \Im \Phi_h(z) \leq 2 \},\mathrm{and} \\
&\Csh_h:=\{z\in \p_h : 1/2 \leq \Re(\Phi_h(z)) \leq 3/2 \: , \: 2\leq \Im \Phi_h(z) \}.
\end{aligned}
\end{equation}
By definition, $\cv_h\in \daron(\C_h)$ and $0\in \partial(\Csh_h)$. 

Assume for a moment that there exists a positive integer $k_h$, depending on $h$, with the following properties:
\begin{itemize}
\item For every integer $k$, with $0\leq k \leq k_h$, there exists a unique connected component of $h^{-k}(\Csh_h)$ 
which is compactly contained in $\Dom h$ and contains $0$ on its boundary. We denote this component by 
$(\Csh_h)^{-k}$. 
\item For every integer $k$, with $0\leq k \leq k_h$, there exists a unique connected component of $h^{-k}(\C_h)$ which has 
non-empty intersection with $(\Csh_h)^{-k}$, and is compactly contained in $\Dom h$. 
This component is denoted by $\C_h^{-k}$. 
\item The sets $\C_h^{-k_h}$ and $(\Csh_h)^{-k_h}$ are contained in 
\[\{z\in\p_h \mid  \frac{1}{2}< \Re \Phi_h(z) <\frac{1}{\alf}-\Bk-\frac{1}{2}\}.\] 
\item The maps $h: \C_h^{-k}\to \C_h^{-k+1}$, for $2\leq k \leq k_h$, and $h: (\Csh_h)^{-k}\to (\Csh_h)^{-k+1}$, for $1\leq k \leq k_h$, are univalent. 
The map $h: \C_h^{-1}\to \C_h$ is a degree two branched covering.
\end{itemize}
Let $k_h$ be the smallest positive integer satisfying the above four properties, and define
\[S_h:=\C_h^{-k_h}\cup(\Csh_h)^{-k_h}.\]
\begin{figure}[ht]
\begin{center}
 \begin{pspicture}(-.5,1.2)(11.4,9)
\epsfxsize=6.3cm
\rput(3.5,5.9){\epsfbox{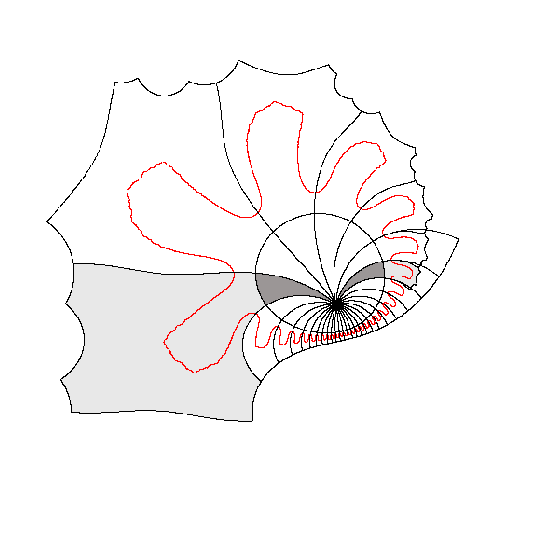}}  
  \psset{xunit=1cm}
  \psset{yunit=1cm}
    \pscurve[linewidth=.6pt,linestyle=dashed,linecolor=black]{->}(5.3,5.1)(5.3,5.5)(5.2,5.8)
    \pscurve[linewidth=.6pt,linestyle=dashed,linecolor=black]{->}(2.2,3.7)(2.3,4.1)(2.2,4.5)
    \pscurve[linewidth=.6pt,linestyle=dashed,linecolor=black]{->}(.7,6.4)(2,6.4)(3.2,6.2)(3.6,5.6)
    \rput(5.3,4.7){$S_h$}
    \rput(2.2,3.4){$\C_h^{-1}$}
    \rput(.1,6.4){$(\Csh_h)^{-1}$}
    \psdots[dotsize=2pt](2.3,5.04)(3.4,4.97)
    \rput(2.,5){\small{$\cp_h$}}
    \rput(3.33,4.77){\small{$\cv_h$}}
\newgray{Lgray}{.99}
\newgray{LLgray}{.88}
\newgray{LLLgray}{.70}
\psdots(7.6,5.5)(8,5.5)(11,5.5)
\pspolygon[fillstyle=solid,fillcolor=LLLgray](7.8,7.3)(7.8,6.1)(8.2,6.1)(8.2,7.3)
\pspolygon[fillstyle=solid,fillcolor=LLgray](7.8,6.1)(7.8,4.9)(8.2,4.9)(8.2,6.1)

\pspolygon[fillstyle=solid,fillcolor=LLLgray](10.2,7.3)(10.05,6.2)(10.45,6.2)(10.6,7.3)
\pspolygon[fillstyle=solid,fillcolor=LLgray](10.05,6.2)(9.8,5.1)(9.9,5)(10,4.97)(10.1,5)(10.2,5.1)(10.45,6.2)
\psdots(7.6,5.5)(8,5.5)
\psline{->}(7.6,5.5)(11.3,5.5)
\psline{->}(7.6,4.7)(7.6,7.3)
\rput(8,5.3){\tiny{$1$}}
\rput(10.9,5.3){\tiny{$\frac{1}{\alf}-\Bk$}}
\rput(7.3,4.9){\tiny{$-2$}}

\psline{->}(6,5.9)(7.5,5.9)
\rput(6.8,6.1){$\Phi$}

\pscurve[linestyle=dashed]{<-}(7.9,6.7)(8.9,6.9)(10.3,6.7)
\rput(9.1,7.1){\tiny{induced map}}

\psline{->}(7.4,4.7)(6.6,3.2)
\rput(7.7,4.){$e^{2\pi \Bi w}$}
\psellipse(6,2.1)(1.2,.9)
\psdots[dotsize=2pt](6,2.1)
\rput(4,2.6){$\rr' (h)$}
\rput(6.2,2.1){\small{$0$}}
\rput(1,8){$h$}
\psline[linewidth=.5pt]{->}(5.7,1.55)(5.85,1.45)(6.05,1.4)
\NormalCoor
\psdot[dotsize=1pt](5.65,1.6)
\psdot[dotsize=1pt](6.1,1.4)
 \end{pspicture}
\caption{The figure shows the sets $\C_h$, $\Csh_h$,..., $\C_h^{-k_h}$, and 
$(\Csh_h)^{-k_h}$. 
The ``induced map'' projects via $e^{2\pi \Bi w}$ to a map $\rr (h)$ defined near $0$.}
\label{sectorpix}
\end{center}
\end{figure}

Consider the map 
\begin{equation}\label{renorm-def}
\Phi_h \circ h\co{k_h} \circ \Phi_h^{-1}:\Phi_h(S_h) \ra \cc. 
\end{equation}
By the Abel functional equation, this map projects via $z=\frac{-4}{27}e^{2 \pi 
\Bi w}$ to a well-defined map $\rr' (h)$, defined on a set containing $0$ in its interior. 
One can see that $\rr (h)$ has asymptotic expansion $e^{2 \pi \frac{-1}{\alf}\Bi}z+ O(z^2)$ near $0$, 
see Figure~ \ref{sectorpix}. 

The conjugate map $s\circ \rr' (h)\circ s^{-1}$, where $s(z):=\bar{z}$ denotes the complex conjugation map, 
is of the form $z \mapsto e^{2 \pi \frac{1}{\alf}\Bi}z+O(z^2)$ near $0$. 
The map  $\rr (h):= s\circ \rr' (h)\circ s^{-1}$, restricted to the interior of $s(\frac{-4}{27}e^{2\pi \Bi(\Phi_h(S_h))})$, 
is called the \textit{near-parabolic renormalization} of $h$ by Inou and Shishikura. 
We simply refer to it as the \textit{renormalization} of $h$. 
It is clear that one time iterating $\rr(h)$ corresponds to several times iterating $h$, through the change of 
coordinates, see Lemma~\ref{renorm}. 
For some applications of a closely related renormalization (Douady-Ghys renormalization) one may 
see \cite{Do86,Do94,Yoc95,Sh98,ABC04} and the references therein.

The following theorem \cite[Main theorem 3]{IS06} states that this definition of renormalization $\rr$ 
can be carried out for perturbations of maps in $\ff$. 
In particular, this implies the existence of $k_h$ satisfying the four properties listed in the definition of 
the renormalization. 
(See \cite[Proposition 13]{BC06} for a detailed argument on this, \footnote{The sets $\C_h^{-k}$ and 
$(\Csh_h)^{-k}$ defined here are (strictly) contained in the sets denoted by $V^{-k}$ 
and $W^{-k}$ in \cite{BC06}. 
The set $\Phi_h(\C_h^{-k}\cup (\Csh_h)^{-k})$ is contained in the union 
\[D^\sharp_{-k} \cup D_{-k} \cup D''_{-k} \cup D'_{-k+1} \cup D_{-k+1} \cup D^\sharp_{-k+1}\] 
in the notation used in \cite[Section 5.A]{IS06}.}.)

Define 
\begin{equation}\label{V}
V:=P^{-1}(B(0,\frac{4}{27}e^{4\pi}))\setminus((-\infty,-1]\cup B)
\end{equation}
where $B$ is the component of $P^{-1}(B(0,\frac{4}{27}e^{-4\pi}))$ containing $-1$ (see Figure~\ref{poly}). 

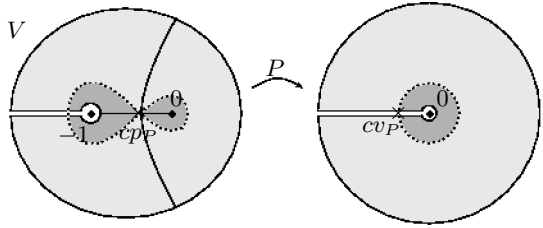
\begin{figure}[ht]
\begin{center}
  \begin{pspicture}(8,3.2)
  \rput(4.5,1.6){\epsfbox{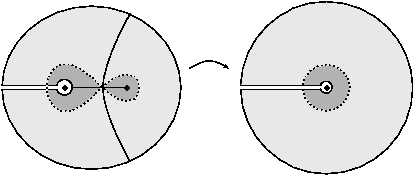}}
      \rput(6.12,1.6){{\small $\times$}}
      \rput(5.9,1.4){{\small $cv_P$}}
      \rput(6.7,1.8){{\small $0$}} 
      \rput(2.7,1.6){{\small $\times$}}
      \rput(2.7,1.3){{\small $cp_P$}}
      \rput(3.2,1.8){{\small $0$}}
      \rput(1.85,1.3){{\small $-1$}}
      \rput(4.5,2.2){{\small $P$}}
      \rput(1.1,2.7){{\small $V$}} 
\end{pspicture}
\caption{A schematic presentation of the polynomial $P$; its domain, and its range. 
Similar colors and line styles are mapped on one another.}
\label{poly}
\end{center}
\end{figure}
By an explicit calculation (see \cite[Proposition 5.2]{IS06}) one can see that the closure of $U$ is contained in the interior of $V$.   
\begin{thm}[Inou-Shishikura]\label{Ino-Shi2} 
There exists a constant $\alf^*>0$ such that if $h \in \ff_\alf\cup\{P_\alf\}$ with $\alf \in (0,\alf^*]$, then $\rr(h)$ is 
well-defined and belongs to the class $\ff_{1/\alf}$. 
Moreover, with the representation $\rr(h):=e^{\frac{2\pi}{\alf}\Bi}\cdot P \circ \psi^{-1}$, the map $\psi:U\to \cc$ 
extends to a univalent map on $V$.
\end{thm}
We need the constant $ \Bk''$ defined by the next lemma.
\begin{lem}\label{turning}
There exists a $\Bk''\in \mathbb{Z}$ such that for every $h \in \ff_\alf\cup\{P_\alf\}$ 
with $\alf \in (0,\alf^*]$, $k_h\leq \Bk''$. 
\end{lem}
\begin{proof}
Comparing with the rotation of angle $\alf$, there exists an integer $j\in [\Bk, \Bk+3]$ such that 
for every $w\in \C_h\cup \Csh_h$ close to zero, there exists a unique inverse orbit $w, h^{-1}(w), \dots, h^{-j}(w)$, 
contained in a neighborhood of zero, with $h^{-j}(w)\in \p_h$. 
As $k_h$ is the smallest positive integer with $\C_h^{-k_h}\cup(\Csh_h)^{-k_h}\subseteq \p_h$, we have $k_h\geq \Bk$.

Now let $\gamma:[0,\infty)\to \Phi_h(\p_h)$ be a continuous curve satisfying the following.
\begin{itemize}
\item[--]$\Phi_h \circ h\co{k_h}\circ \Phi_h^{-1}(\gamma(t))= s+ t\Bi$, for $t\in (-2,\infty)$, and a fixed 
$s\in [1/2,3/2]$;
\item[--]$\diam \{\Re \gamma(t)\mid t\in [0,\infty)\}\geq k_h-\Bk-4$.  
\end{itemize}
The curve $\hat{\gamma}:=\ex \circ \gamma$ lands at zero,  and part of it must spiral at least $k_h-\Bk-4$ 
times around zero. 
By the definition of renormalization, $\rr (h)$ maps $\hat{\gamma}$ to a straight ray landing at zero. 
On the other hand, by Theorem~\ref{Ino-Shi2}, $\rr (h)$ is of the form $e^{2\pi\alf \Bi}\cdot P\circ \psi: U\to \cc$ with
$\psi$ extending univalently 
over the larger domain $V$ which compactly contains $U$. 
By the Koebe distortion Theorem, the total spiraling of the pull back of a straight ray landing at zero under $\rr (h)$ must be 
uniformly bounded by some constant independent of $h$. 
That is, $k_h-\Bk-4$ is uniformly bounded from above.
\end{proof}

We also need the following lemma in the sequel. 
\begin{lem}\label{L:extra-space}
There exists a positive constant $\delta_1\leq 1/8$ such that for every $h$ in $\ff_\alf\cup\{P_\alf\}$ 
with $\alf \in (0,\alf^*]$ we have
\begin{itemize}
\item $\forall j \in \mathbb{Z},\; \ex(B(j,\delta_1)) \subset \daron (\C_h),$
\item for every $\xi\in\cc$ with $\ex(\xi)\in \cup_{j=0}^{k_h+\lfloor 1/\alpha\rfloor-\Bk-1}h\co{j}(S_h)$, 
$\ex (B(\xi,\delta_1))\subset \Dom h$.
\end{itemize}
\end{lem}
\begin{proof}
The sets $\C_h^{-i}\cup (\Csh_h)^{-i}$, for $i=0,1,2,\dots k_h$, are compactly contained in $\Dom h$, 
therefore,  $\cup_{j=0}^{k_h+\lfloor 1/\alpha\rfloor-\Bk-1}h\co{j}(S_h)$ is compactly contained in $\Dom h$. 
By the compactness of the class $\ff$, and the continues dependence of the Fatou coordinates on the map, as well as 
Lemma~\ref{turning}, there exists $\delta>0$ such that 
\[B(-4/27, \delta)\subseteq \C_h, \text{ and }
B(\cup_{j=0}^{k_h+\lfloor 1/\alpha\rfloor-\Bk-1}h\co{j}(S_h), \delta)\subseteq \Dom h.\]
Both statements in the lemma follow from the above inclusions. 
\end{proof}

Let $[0; a_1,a_2, \dots]$ denote the continued fraction expansion of $\alf$ as 
\[\alf=\cfrac{1}{a_1+\cfrac{1}{a_2+\cfrac{1}{a_3+\dots} }}\] 
Define $\alf_0:=\alf$, and inductively for $i\geq1$ define $\alf_i\in (0,1)$ as 
\[\alf_i:=\frac{1}{\alf_{i-1}}\;  (\textrm{ mod } 1).\] 
For every $i\geq 0$, $a_{i+1}:=\lfloor 1/\alpha_i\rfloor$ and $\alf_i=[0; a_{i+1}, a_{i+2}, \dots]$. 

If we fix a constant $N\geq 1/\alf^*$, then $\alf\in \irr$ implies that $\alf_{j}\in (0,\alf^*)$, for $j=0,1,2,\dots$.  
We use this constant $N$ in the rest of this paper.

Let $\alf\in \irr$ and define $f_0:=P_\alf$. 
Then, using Theorem~\ref{Ino-Shi2}, inductively define the sequence of maps 
\[f_{n+1}:=\rr (f_n): \Dom f_{n+1} \to \cc.\]
For every $n$,  
\[f_n:\Dom f_n \to \cc,\; f_n(0)=0,\; \text{ and }  f_n'(0)= e^{2\pi \alf_n \Bi}.\]
\section{The renormalization tower}\label{sec:tower}
\subsection{Changes of coordinates and sectors around the fixed point}
\begin{rem}
To simplify the technical details of exposition, we assume that 
\begin{equation}\label{alf*-1}
N\geq \Bk+\hat{\Bk}+1.
\end{equation}
The reason for this is to make $\Phi_{f_n}(\p_{f_n})$ wide enough to contain a set that will be 
defined in a moment. 
However, one can avoid this condition by extending $\Phi_{f_n}$ and $\Phi_{f_n}^{-1}$ to 
larger domains, using the dynamics of $f_n$. 
See Subsection~\ref{SS:extensions}.
\end{rem}
For $n\geq 0$, let $\Phi_n:=\Phi_{f_n}$ denote the Fatou coordinate of $f_n\colon \Dom f_n \ra \cc$ 
defined on the set $\p_n:=\p_{f_n}$. 
For our convenience we define the covering map  
\[\ex(\zeta):= \zeta \longmapsto \frac{-4}{27} s (e^{2\pi \Bi \zeta}):\cc \ra \cc^*, \text{ where } s(z)=\bar{z}.\]
By part ii of Theorem~\ref{Ino-Shi1}, Inequality~\eqref{alf*-1}, and that $\p_n$ is simply connected, there is an inverse branch of $\ex$ that maps $\p_n$ into the range of $\Phi_{n-1}$; 
\[\eta_n: \p_n\to \Phi_{n-1}(\p_{n-1}).\]
There may be several choices for this map but we choose one of them (for each $n$) with
\begin{equation*}
\Re (\eta_n(\p_n))\subset [0,\hat{\Bk}+1], 
\end{equation*}
and fix this choice for the rest of this note. 
Now, define 
\[\psi_n :=\Phi_{n-1}^{-1} \circ \eta_n  \colon \p_{n}\ra \p_{n-1}.\] 
Each $\psi_n$ extends continuously to $0\in \partial \p_n$ by mapping it to $0$. 
Let $\Psi_1:=\psi_1$, and for $n\geq 2$ form the composition
\[\Psi_n:=\psi_1 \circ \psi_2 \circ \dots \circ \psi_n\colon \p_n \ra \p_0.\] 

For every $n\geq 0$, let $\C_n$ and $\Csh_n$ denote the corresponding sets for $f_n$ defined in \eqref{sector-def} 
(i.e., replace $h$ by $f_n$). 
Denote by $k_n$ the smallest positive integer with  
\[S_n^0:=\C_n^{-k_n}\cup (\Csh_n)^{-k_n}\subset \{z\in \p_n \mid 0< \Re \Phi_n(z)< \lfloor \frac{1}{\alf_n}\rfloor -\Bk-1\}.\]
For every $n\geq 0$ and $i\geq 2$, define the sectors 
\begin{gather*}
S_n^1:=\psi_{n+1}(S_{n+1}^0)\subset \p_n, \text {and} \\
S_n^i:=\psi_{n+1}\circ \dots \circ\psi_{n+i}(S_{n+i}^0)\subset \p_n.
\end{gather*}
By definition, the critical value of $f_n$ is contained in $f_n\co{k_n}(S_n^0)$.
Also, all these sectors contain $0$ on their boundaries. 
We will mainly work with $S^i_0$, for $i\geq 0$. 
\begin{lem}\label{renorm}
Let $z\in\p_n$ be a point with $w\!\!:=\ex\circ \Phi_n(z)\in \Dom f_{n+1}$.
There exists an integer $\ell_z$ with  $1\leq \ell_z \leq \lfloor 1/\alf_n\rfloor-\Bk+k_n-1$, 
such that
\begin{itemize}
\item the finite orbit $z,f_{n}(z),f_{n}\co{2}(z),\dots,f_{n}\co{\ell_z}(z)$ is defined, and $f_{n}\co{\ell_z}(z)\in \p_{n}$;
\item $\ex\circ \Phi_{n}(f_{n}\co{\ell_z}(z))= f_{n+1}(w)$.
\end{itemize}
\end{lem}
\begin{proof}
As $w\in \Dom f_{n+1}$, by the definition of renormalization $\rr(f_{n})=f_{n+1}$, there are 
$\zeta \in \Phi_{n}(S^0_{n})$ and $\zeta' \in \Phi_{n}( \C_{n}\cup\Csh_{n})$, 
such that
\begin{equation*} 
\ex(\zeta)=w,\quad \ex(\zeta')=f_{n+1}(w),\text { and} \quad \zeta'=\Phi_{n}\circ f_{n}\co{k_{n}}\circ\Phi_{n}^{-1}(\zeta).
\end{equation*}
Since $\ex (\Phi_{n}(z))=w$, there exists an integer $\ell$ 
with 
\[-k_{n}+1\leq\ell\leq \lfloor 1/\alf_{n}\rfloor-\Bk-1,\] 
such that $\Phi_{n}(z)+\ell=\zeta$.

By the Abel functional equation for $\Phi_{n}$, we have
\begin{align*}
\zeta' &=\Phi_{n}\circ f_{n}\co{k_{n}}\circ\Phi_{n}^{-1}(\zeta)\\
          &=\Phi_{n}\circ f_{n}\co{k_{n}}\circ\Phi_{n}^{-1}(\Phi_{n}(z)+\ell)\\
          &=\Phi_{n}\circ f_{n}\co{k_{n}+\ell}(z).
\end{align*}

Letting $\ell_z:=k_{n}+\ell$, we have 
\[1 \leq  \ell_z \leq k_{n}+ \lfloor 1/\alf_{n} \rfloor-\Bk-1, \quad f_{n}\co{\ell_z}(z)=\Phi_{n}^{-1}(\zeta') \in \p_{n}, \text{ and}\] 
\begin{align*} 
\ex\circ\Phi_{n}(f_{n}\co{\ell_z}(z))&=\ex \circ \Phi_{n}(\Phi_{n}^{-1}(\zeta'))\\
&=\ex (\zeta')\\
&=f_{n+1}(w).
\qedhere
\end{align*}
\end{proof}

It is clear that in the above lemma there are many choices for $\ell_z$. 
By making some specific choices for $\ell_z$ and inductively using the above lemma, the number of iterates on 
level $n$ are related to the number of iterates on level $0$

Define 
\[\p_n':=\{w\in\p_n\mid 0 <\Re \Phi_{n}(w) <\lfloor 1/\alf_n\rfloor-\Bk-1\}.\]
\begin{lem}\label{conjugacy}
For every $n\geq 1$ we have 
\begin{itemize}
 \setlength{\itemsep}{.2em}
\item[i.]for every $w\in \p_n'$, $f_0\co{q_n}\circ\Psi_n(w)=\Psi_n \circ f_n(w)$,
\item[ii.]for every $w\in S^0_n$, $f_0\co{(k_nq_n+q_{n-1})}\circ\Psi_n(w)=\Psi_n\circ f_n\co{k_n}(w)$.
\end{itemize}
\end{lem}
\begin{proof}
By the previous lemma, an iterate of $f_n$ on $\p_n$ corresponds to some iterate of $f_0$ on $\Psi_n(\p_n)$. 
The correct numbers of iterates, $q_n$ and $k_nq_n+q_{n-1}$, when $|w|$ is small, are determined by 
comparing the map to the rotation of angle $\alf$. 
Hence, the equalities hold at points near $0$ and therefore, must hold over the domain of definitions by the 
analytic continuation.  
A more detailed proof may be found in \cite{Ch10-I}.   
\end{proof}
\subsection{The approximating neighborhoods}
For every $n\geq 0$, consider the union
\begin{equation}\label{union-0}
\Omega^0_n:=\bigcup_{i=0}^{k_n+\lfloor 1/\alf_n\rfloor-\Bk-1}f_n\co{i}(S^0_n)\subset \Dom f_n.
\end{equation}
Using Lemma~\ref{conjugacy}, we transfer the iterates in the above union to the dynamic plane of $f_0$ to form 
\[\Omega^n_0:=\bigcup_{i=0}^{q_n(k_n+\lfloor 1/\alf_n\rfloor-\Bk-1)+q_{n-1}} f_0\co{i}(S^n_0)\bigcup \{0\}.\]

The upper bound in the above union is obtained as follows. 
The first $k_n$ iterates in \eqref{union-0} correspond to $k_nq_n+q_{n-1}$ number of iterates on level $0$ by 
Lemma~\ref{conjugacy}-ii, and the remaining $\lfloor 1/\alf_n\rfloor-\Bk-1$ iterates in \eqref{union-0} 
amounts to $q_n(\lfloor 1/\alf_n\rfloor-\Bk-1)$ number of iterates by Lemma~\ref{conjugacy}-i. 
With these choices of the number of iterates,  the unions satisfy the following property.
\begin{propo}\label{neighbor}
For every $n\geq0$, 
\begin{itemize}
\item[i.] $\Omega_{0}^{n+1}$ is compactly contained in the interior of $\Omega_0^n$, 
\item[ii.] $\pc(f_0)$ is contained in the interior of $\Omega^n_0$. 
\end{itemize}
\end{propo}
\begin{proof}The proof is broken into several steps.
\medskip

{\em Step 1.} $\forall n\geq 0$,  $\Omega^{n+1}_0 \subset \Omega^n_0$.

To prove this, it is enough to show the following claim. 
For every $z \in S^{n+1}_0$ there are points $z_1,z_2,\dots, z_{m}$ in 
$S^n_0$ as well as non-negative integers $t_1,t_2,\dots,t_{m+1}$, for some positive integer $m$, 
fulfilling the following properties:
\begin{itemize}
\item $f_0\co{t_1}(z_1)=z$  and $f_0\co{t_{m+1}}(z_m)=f_0\co{(q_{n+2}+(k_{n+1}-\Bk-1)q_{n+1})}(z)$,
\item $f_0\co{t_j}(z_{j-1})=z_j$, for $j=2,3,\dots,t_{m-1}$, 
\item $t_j\leq q_{n+1}+(k_n-\Bk-1)q_n$, for every $j=1,2,\dots,m+1$. 
\end{itemize} 

Given $z\in S^{n+1}_0$, let $\zeta:=\Psi_{n+1}^{-1}(z)\in S^0_{n+1}$. 
By the definition of $S^0_{n+1}$, the iterates 
\[\zeta, f_{n+1}(\zeta),f_{n+1}\co{2}(\zeta),\dots, f_{n+1}\co{(k_{n+1}+\lfloor \frac{1}{\alf_{n+1}}\rfloor-\Bk-1)}(\zeta)\] 
are defined and belong to the domain of $f_{n+1}$. 

By Lemma~\ref{renorm} for $\psi_{n+1}(\zeta)$, there are two points $\xi_1:=\psi_{n+1}(\zeta)$ 
and $\xi_2$ in $\p_n$ as well as a positive integer $\ell_0$ with $\ex\circ \Phi_n(\xi_1)=\zeta$, 
$\ex\circ \Phi_n(\xi_2)=f_{n+1}(\zeta)$ and $f_n\co{\ell_0}(\xi_1)=\xi_2$. 
Let $\sigma_1\in S_n^0$ and an integer $\ell_1$ with $1\leq\ell_1\leq k_n+a_{n+1}-\Bk-1$ be such that 
$f_n\co{\ell_1}(\sigma_1)=\xi_1$. 
By the same lemma, there is a point $\sigma_2$ in the orbit 
\[\xi_1, f_n(\xi_1), f_n\co{2}(\xi_1),\dots,f_n\co{(\ell_0-1)}(\xi_1),\xi_2\]
which belongs to $S_n^0$. 
Let $\ell_2$ be the positive integer with $1 \leq\ell_2\leq k_n+\lfloor \frac{1}{\alf_n}\rfloor -\Bk-1$ such that 
$f_n\co{\ell_2}(\sigma_1)=\sigma_2$.

By the same argument for $\xi_2$ with $\ex\circ \Phi_n (\xi_2)=f_{n+1}(\zeta)$, we obtain points 
$\sigma_3$ in $S_n^0$, $\xi_3$ in $\p_n$ and a positive integer $\ell_3$ with $1 \leq\ell_3\leq k_n+a_{n+1}-\Bk-1$, 
such that $f_n\co{\ell_3}(\sigma_2)=\sigma_3$. 

Repeating this argument with $\xi_3,\xi_4,\dots,\xi_m$, 
where $m=k_{n+1}+\lfloor \frac{1}{\alf_{n+1}}\rfloor-\Bk-1$, one obtains a sequence 
$\sigma_1, \sigma_2, \dots, \sigma_m$
of points in $S_n^0$ and positive integers 
$\ell_1,\ell_2,\dots,\ell_{m+1}$,
all bounded by $k_n+\lfloor \frac{1}{\alf_n}\rfloor-\Bk-1$, which satisfy
\begin{itemize}
\item $f_n\co{\ell_{j+1}}(\sigma_j)=\sigma_{j+1}$ for all $j=2,3, \dots,m-1$
\item $f_n\co{\ell_1}(\sigma_1)=\xi_1$ and $f_n\co{\ell_{m+1}}(\sigma_m)=\xi_m$.  
\end{itemize}

Now we define $z_i:=\Psi_n(\sigma_i)\in S_0^n$, for $i=1,2,\dots,m$. 
By definition, $\Psi_n(\xi_1)=z$. 
We claim that $\Psi_n(\xi_m)=f_0\co{(q_{n+2}+(k_{n+1}-\Bk-1)q_{n+1})}(z)$. 
This is because 
\[\ex\circ \Phi_n(\xi_m)=f_{n+1}\co{(k_{n+1}+\lfloor \frac{1}{\alf_{n+1}}\rfloor-\Bk-1)}(\zeta),\] 
which is mapped to $f_0\co{(q_{n+2}+(k_{n+1}-\Bk-1)q_{n+1})}(z)$ under $\Psi_{n+1}$, 
using Lemma~\ref{conjugacy}.  

By Lemmas~\ref{conjugacy}, $\ell_j$ times iterating $f_n$ corresponds to $t_j$ times iterating $f_0$, for each 
$j=1,2,\dots,\ell_{m+1}$. 
With the same lemma, as $\ell_j$ is bounded by $k_n+\lfloor \frac{1}{\alf_n}\rfloor-\Bk-1$, 
each $t_j$ is bounded by 
\[k_nq_n+q_{n-1}+(\lfloor \frac{1}{\alf_n}\rfloor-\Bk-1)q_n\leq q_{n+1}+q_n(k_n-\Bk-1).\] 

{\em Step 2.} $\forall n\geq 0$, $\Omega^{n+1}_0$ is compactly contained  in the interior of $\Omega^n_0$.

To prove this we use the open mapping property of holomorphic and anti-holomorphic mappings. 
If $z'$ is a point in the closure of $\Omega^{n+1}_0$, as $f_0$ is a polynomial defined on the whole complex 
plane, there exists a point $z$ in the closure of $S^{n+1}_0$ with $f_0\co{t_0}(z)=z'$ , for a non-negative integer 
$t_0$ less than $q_{n+2}+(k_n-\Bk-1)q_{n+1}$. 
By the proof of Lemma~\ref{renorm}, all the points $\sigma_j$ in the above argument can be chosen to be
in the interior of $S_n^0$. 
Hence, all $z_i$ are contained in the interior of $S_0^n$.     
\medskip

{\em Step 3.} $\forall n\geq 0$, the critical point of $f_0$ belongs to $\Omega^n_0$ and 
can be iterated at least $(a_{n+1}-\Bk-1)q_n$ times within this set. 

The map $f_n\colon S_n^0\ra f_n\co{k_n}(S_n^0)$ is a degree two branched covering. 
Thus, by Lemma~\ref{conjugacy}, $f_0\co{(k_nq_n+q_{n+1})}\colon  S_0^n\ra \Psi_n(f_n\co{k_n}(S_n^0))$ is 
also a degree two map. 
That means that the critical point of $f_0$ is contained in the union 
$\cup^{k_nq_n+q_{n-1}}_{i=0}  f_0\co{i}(S_0^n)$. 
Therefore, by the definition of $\Omega^n_0$, the critical point can be iterated at least 
\[q_{n+1}+(k_n-\Bk-1)q_n-k_nq_n-q_{n-1}=(a_{n+1}-\Bk-1)q_n\] 
times within $\Omega^n_0$.  
 
As $a_{n+1}-\Bk-1\geq1$ and $q_n$ grows (exponentially) to infinity as $n$ goes to infinity, 
the critical point of $f_0$ can be iterated infinite number of times within each $\Omega^n_0$.  

Finally, for every $n\geq0$, $\Omega^n_0$ contains the closure of $\Omega^{n+1}_0$ in its interior, and 
$\Omega^{n+1}_0$ contains the orbit of the critical point. 
Therefore, the post-critical set must be contained in every $\Omega^n_0$.
\end{proof}
\section{Area of the post-critical set}\label{sec:measure}
In this section we prove the following proposition which combined with Proposition~\ref{neighbor} 
implies Theorem~\ref{pc-area}.
Let $N$ be the constant introduced at the end of Section~1.
\begin{propo}\label{intersection-area}
For all Brjuno $\alf$ in $\irr$,  $(\cap_{n=0}^\infty \Omega^n_0)\cap J(P_\alf)$ has zero area.
\end{propo}
\subsection{Natural extensions}\label{SS:extensions}
For every $n\geq 1$, let $\textrm{Fil}(\Omega^0_n)$ denote the set obtained from adding the bounded components of 
$\cc\setminus \Omega_n^0$ to $\Omega_n^0$, if there is any. 
For all integer $n\geq 1$ and $r\in [0, 1/\alf_n-\Bk]$, let $I_{n,r}$ denote the connected 
component of 
\[\textrm{Fil}(\Omega_n^0) \cap \Phi_n^{-1}\{r+t\Bi : t\in \mathbb{R}\},\]
containing zero on its boundary. 
Each $I_{n,r}$ is a smooth curve in $\textrm{Fil}(\Omega^0_n)$ connecting the boundary of $\Omega^0_n$ to $0$. 
This implies that there is a continuous inverse branch of $\ex$ defined on every 
$\Omega^0_n\setminus I_{n,r}$.

By Theorem~\ref{Ino-Shi1}-ii, Lemma~\ref{turning},  and  the precompactness of the class $\fff$, there exists an 
integer $\Bk'$ such that 
\begin{equation}\label{E:k'-spiral}
\forall n\geq 1 \text{ and } \forall r\in [0,1/\alf_n-\Bk], \sup_{z,w\in \Omega^0_n \setminus I_{n,r}} 
|\arg(w)-\arg(z)| \leq 2\pi \Bk',
\end{equation}
for any continuous branch of argument defined on $\Omega^0_n \setminus I_{n,r}$.  

\begin{rem}
To avoid un-necessary technical details in the manuscript, from now on we assume that  
\begin{equation}\label{E:k'-condition}
N\geq 2\Bk'+\Bk+1.
\end{equation}
This guarantees that on every level $n$, $1/\alpha_n-\Bk\geq 2\Bk'+1$ and hence one can embed a 
lift $\ex^{-1}(\Omega_n^0\setminus I_{n, r})$ into some $\Omega_{n-1}\setminus I_{n-1, r'}$,  
via the Fatou coordinates. 
\end{rem}

The Fatou coordinate $\Phi_h$, and its inverse, have natural extensions onto some larger domains.
In the forward direction it is achieved as follows.
First define the set $\p_h^l$ as follows:
\begin{itemize}
\item if $\Bk'\leq k_h$
\[\p_h^l:= \bigcup_{j=0}^{2\Bk'} h\co{(k_h-\Bk'+j)}(S_h),\]
\item if  $\Bk'>k_h$
\[\p_h^l:=\bigcup_{j=0}^{\Bk'} h\co{(k_h+j)}(S_h) \bigcup \bigcup_{j=1}^{\Bk'-k_h}\Phi_h^{-1}(\Phi_h(S_h)-j).\]
\end{itemize}

Then $\Phi_h^l: \p_h^l\to \cc$ is defined as 
\begin{equation}\label{E:left-extension}
 \Phi_h^l(z):= \Phi_h(h\co{j}(z))-j,
\end{equation}
where $j$ is the smallest non-negative integer with $h\co{j}(z)\in \p_h$. 
The extended map is not univalent anymore, and has critical points at preimages of the critical point of $h$. 
These critical points of the extended map are sent to the points $0,-1,-2,\dots, -\Bk'+1$.

The inverse map, which is denoted by $\phi_h^\dagger$, is defined on the set
\begin{equation}\label{E:W}
\mathcal{W}_h:=\{w\in \cc\mid \Re w \in (\Bk', \alpha(h)^{-1}-\Bk)\}\bigcup \cup_{j=0}^{k_h+\Bk'}\Phi_h(S_h)+j.
\end{equation}
If $\Re w\in (\Bk', 1/\alpha-\Bk)$, $\Phi_h^\dagger(w):=\Phi_h^{-1}(w)$, and when $w\in \Phi_h(S_h)+j$
\begin{equation}\label{E:right-extension}
\Phi_h^\dagger(w):=h\co{j}\circ \Phi_h^{-1}(z-j)
\end{equation}
By the equivariance property of the map $\Phi_h$, one can see that this provides us with a well-defined map.
The map $\Phi_h^\dagger$ has critical points that are mapped to 
\[\cv_h, h(\cv_h), \dots ,h\co{\Bk_h-1}(\cv_h).\]
As usual, let $\Phi_n^\dagger:=\Phi_{f_n}^\dagger$ and $\Phi_n^l:=\Phi_{f_n}^l$, for $n=0,1,2,\dots$.
\subsection{The sequence $\zeta_n$}\label{S:heights}
To each non-zero $z_0\in \cap^{\infty}_{n=0}\Omega^n_0\cap J(P_\alf)$, we associate a sequence of pairs 
\begin{align}\label{E:pairs}
(z_i,\zeta_i), i=0,1,2, \dots.
\end{align}
(The sequence $\{z_i\}$ is the same as the first coordinates of the sequence of quadruples 
considered in~\cite{Ch10-I}). 
For $n\geq 0$, define the sets
\begin{gather*}\label{E:Three-sets}
\mathscr{A}_n:=\{z\in \p_n\mid \Phi_n(z)\in \cup_{j=1}^{\Bk'} B(j,\delta_1), \text{ or } 
\Re \Phi_n(z)\in (\Bk'+1/2, 1/\alpha_n-\Bk)\}, \\
\mathscr{B}_n:=\Omega_n^0\setminus \mathscr{A}_n,
\end{gather*}
where $\delta_1$ is obtained in Lemma~\ref{L:extra-space}. 
See Figure~\ref{GoDown}.

Sequence \eqref{E:pairs} is inductively defined as follows. 
To define $\zeta_0$, 
\begin{itemize}
\item if $z\in \mathscr{A}_0$, let $\zeta_0:= \Phi_h(z)$, 
\item if $z\in \mathscr{B}_0$, let $\zeta_0$ be a preimage of $z$ under $\Phi_h^\dagger$. 
\end{itemize}
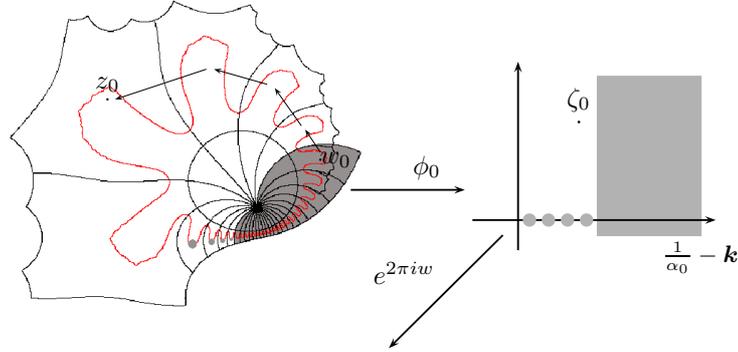
\begin{figure}
\begin{center}
\begin{pspicture}(0,1.6)(10,7)
\epsfxsize=6cm
\rput(3,4){\epsfbox{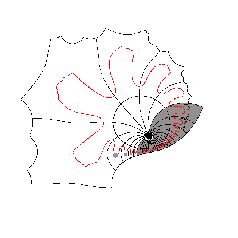}}
\psdot[dotsize=.9pt](1.8,5)
\rput(1.8,5.2){$z_0$}
\psdot[dotsize=.9pt](4.6,4.2)
\rput(4.8,4.2){$w_0$}

\psline[linearc=.04, linewidth=.3pt]{->}(4.6,4.3)(4.4,4.6)
\psline[linearc=.04, linewidth=.3pt]{->}(4.35,4.65)(4.,5.15)
\psline[linearc=.04, linewidth=.3pt]{->}(3.9,5.2)(3.2,5.4)
\psline[linearc=.04, linewidth=.3pt]{->}(3.1,5.4)(1.9,5)
\psline[linewidth=.7pt]{->}(5.,3.8)(6.5,3.8)
\rput(6,4.1){$\phi_0$}
\pspolygon[fillstyle=solid,fillcolor=LLLgray,linecolor=LLLgray](8.25,3.2)(9.6,3.2)(9.6,5.3)(8.25,5.3)
\psline[linewidth=.7pt]{->}(6.6,3.4)(9.8,3.4)
\psline[linewidth=.7pt]{->}(7.2,3.)(7.2,5.5)
\rput(8,5){$\zeta_0$}
\psdot[dotsize=1pt](8,4.7)
\psline[linewidth=.7pt]{->}(7,3.2)(5.5,1.7)
\rput(5.7,2.7){$e^{2\pi i w}$}
\psset{linecolor=LLLgray}
\qdisk(7.35,3.4){2.5pt}
\qdisk(7.6,3.4){2.5pt}
\qdisk(7.85,3.4){2.5pt}
\qdisk(8.1,3.4){2.5pt}
\rput(9.6,2.9){{\small $\frac{1}{\alf_0}-\Bk$}}
\end{pspicture}
\label{GoDown}
\caption{The two different colors correspond to the two different ways of going down the renormalization tower.
The gray part corresponds to $\mathscr{A}$ and the rest to $\mathscr{B}$.}
\end{center}
\end{figure}
The point $\zeta_0$ satisfying this property is not necessarily unique. 
However, one can choose any of them. 
Indeed, when $|z_0|$ is small enough, $\zeta_0$ is uniquely determined.
Otherwise, there are at most two choices for it. 

Now, define $z_1:=\ex(\zeta_0)$, and note that since $z_0 \in \Omega^1_0$, $z_1 \in \Omega^0_1$. 
Thus, we can repeat the above step to define the next pair $(z_1, \zeta_1)$, and so on. 
In general, for $j\geq 0$,
\begin{equation}\label{sign-property}
\begin{gathered} 
z_{j+1}=\ex(\zeta_{j}), \text{ and  } \\
\Phi_j^l(z_j):=\zeta_j \text{ or } \Phi_j^\dagger (\zeta_j)=z_j.
\end{gathered}
\end{equation}
Note that by the definition of this sequence, 
\begin{equation}\label{E:place-of-zeta}
\forall l\geq 0, \Bk'+1/2 \leq \Re \zeta_l \leq \alf_l^{-1}-\Bk+k_l+\Bk' \;\text{ or }\; 
\zeta_l\in \cup_{j=1}^{\Bk'} B(j,\delta_1).    
\end{equation}
 
By 1/4-theorem, every map $h\in \fff$ is univalent on $B(0,1/12)$.  
Let 
\[D:=1+\frac{1}{2\pi} \log \frac{16}{9},\]
and note that 
\begin{equation}\label{constant-D2}
 D\geq \delta_1 \text{ and } \ex (D \Bi)\leq 1/12.
\end{equation}
Since $\rr h$ also belongs to $\fff$, by the definition of renormalization, $h^{k_h}$ is univalent on 
$\{z\in S_h|\Im \Phi_h(z)\geq D\}$.  
With this choice of $D$, we decompose the intersection into two sets  
\begin{gather*} 
\Gamma:=  \{z \in \cap_{n=0}^\infty \Omega^n_0\cap J(P_\alf) \mid \exists N \text{ such that for all $m\geq N$,}
 \Im \zeta_m \geq D/\alf_m \},\\
 \Lambda:= \{z \in \cap_{n=0}^\infty \Omega^n_0\cap J(P_\alf) \mid\text{there are infinitely many $m$ with } 
 \Im \zeta_m < D/\alf_m \}.
\end{gather*} 
The areas of these two sets are treated in the following two subsections.
\subsection{Area of $\Gamma$}\label{SS:non-porous-part}
The set $\Gamma$ is contained in the union of the sets 
\[\Gamma^m:= \{z\in \Gamma\mid \text{ for all } n\geq m,\, \Im\zeta_n\geq D/\alf_n\}, 
\text{ for } m=0,1,2,\dots.\]
Every $\Gamma^m$ projects (under piecewise holomorphic or anti-holomorphic maps) 
onto the set $\Gamma_m$ on the dynamic plane of level $m$; defined as
\[\Gamma_m:= \{z_m=\ex(\zeta_{m-1)})\mid z\in \Gamma^m \}.\]  

Fix an $m\in \{0,1,2,\dots \}$, and for $j\geq m$ let 
\[\Pi_j:=\{\zeta_j \mid z\in \Gamma^m \}.\]
Note that here $\zeta_j$ is uniquely determined by $z_j$, by our choice of the constant $D$.
Projecting the above set onto the round cylinder $\mathbb{C}/\mathbb{Z}$, we define 
\[B_j:=\Pi_j/\mathbb{Z}, \text{ for } j=m, m+1,m+2,\dots.\]
That is, $B_j$ is the set of all points $w\in \mathbb{C}/\mathbb{Z}$ for which there exists an integer 
$i$ with $w+i\in \Pi_j$.
Each $B_j$ lies above the line $\Im w=D/\alf_j$. 
We would like to show that these sets have zero area. 
To prove this, we first show that the sequence,
\[b_j:= \sup_{\zeta, \zeta'\in B_j} \Im (\zeta-\zeta'), \text{ for } j=m, m+1,m+2,\dots,\]
satisfies the inequality in the following lemma.
\begin{lem}\label{recursive-heights}
There exists a constant $A$ such that for every $j> m$ we have 
\[b_{j-1}\leq \alf_{j}b_j+ A.\]
\end{lem}
To prove the above lemma, we need an estimate on the derivative of the change of coordinates 
\begin{equation}\label{chi-h}
\chi_h:= \mathbb{L}og \circ \Phi_h^\dagger,
\end{equation}
where $\mathbb{L}og$ is an arbitrary inverse branch of $\ex$ defined on a slit neighborhood of $0$. 

For every $r\in (0,1/2]$ and $\alf\in (0,1)$ define the set 
\[\Theta(r,\alf):=\{w\in\mathbb{C} \mid \Im w\geq-2/\alf\}\setminus  B(\mathbb{Z}/\alf, r/\alf).\]
\begin{propo}\label{main-estimate}
There exists a constant $C$ such that for every $r\in (0,1/2]$, every $h\in \fff$, and every 
$w\in \Dom \chi_h \cap \Theta(r,\alf)$, we have
\begin{equation*}
|\chi' _h(w)-\alf| \leq C\frac{\alf}{r} e^{-2\pi \alf \Im w}.
\end{equation*} 
\end{propo} 
This lemma is proved in Section~\ref{S:perturbed-coordinate}. 

\begin{proof}[Proof of Lemma~\ref{recursive-heights}]
Let $\chi_n:=\chi_{f_n}$, for $n=0,1,2,\dots$, denote the map defined by \eqref{chi-h} when $h=f_n$.
Since $\Im \zeta_j\geq D/\alf_j\geq \delta_1$, one infers from the definition of the pairs \eqref{E:pairs}
that $\Phi_j^\dagger(\zeta_j):=z_j$.
Therefore, from~\eqref{E:place-of-zeta}, for $j=m, m+1, m+2, \dots$, 
\begin{equation*}\label{half-strip}
\begin{gathered}
\Pi_j\subseteq \{w\in \mathbb{C}\mid \Im w\geq D/\alf_j, \Bk'+1/2\leq \Re w\leq  \Bk'+1/\alf_j-\Bk+k_n\},\\
 \chi_j(\Pi_j)/\mathbb{Z}= B_{j-1}.
\end{gathered}
\end{equation*}
For arbitrary $\zeta, \zeta'\in \Pi_j$, let $\gamma$ be a straight line segment connecting $\zeta'$ to $\zeta$. 
By Proposition~\ref{main-estimate}, with $r=1/2$, we have 
\begin{align*}
|\chi_j(\zeta')-\chi_j(\zeta)-&\alf_j(\zeta'-\zeta)|\\
&=|\int_\gamma \chi_j'-\alf_j \, dw|\\
&\leq \int_\gamma 2 C\alf_j e^{-2\pi \alf_j \Im w} \, dw\\
&\leq 2 C\alf_j e^{-2\pi \alf_j (D/\alf_j)} (\frac{1}{\alf_j}-\Bk+k_n)+ 2 \int_{D/\alf_j}^{\infty} 
C\alf_j e^{-2\pi \alf_j t}\, dt\\
&\leq 2 C (1+\Bk'') e^{-2\pi D}+\frac{e^{-2\pi D}}{\pi}.
\end{align*} 
The result follows from taking the supremum over all $\chi_j(\zeta), \chi_j(\zeta')$, and then all 
$\zeta,\zeta'$.
\end{proof}

\begin{lem}\label{bounded-height}
$\limsup_{j\geq m} b_j<4 A$.
\end{lem}

\begin{proof}
For $j\geq 0$, recall the Brjuno sum, 
\[\mathcal{B}(\alf_j):= \log \frac{1}{\alf_j}+ \alf_j\log \frac{1}{\alf_{j+1}}
+\alf_j\alf_{j+1}\log \frac{1}{\alf_{j+2}}+\cdots.\]  
By a result of Yoccoz \cite{Yoc95} (this also follows from \cite{Ch10-I}), the Siegel disk of $f_j$ 
contains the disk of radius $e^{-\mathcal{B}(\alf_j)+M}$ about the origin, for some universal constant $M$. 
This implies that for every $j\geq m$, we have $b_j\leq \mathcal{B}(\alf_{j+1})/2\pi$.

Fix $j\geq m$, and let $l$ be an arbitrary integer bigger than $j$.
Inductively using the inequality in Lemma~\ref{recursive-heights}, for $j+1, j+2, \dots, l$, we obtain 
\[b_j \leq \alf_{j+1}\alf_{j+2}\cdots \alf_{l}\, b_{l}
+(1+\alf_{j+1}+ \alf_{j+1}\alf_{j+2}+ \dots+ \alf_{j+1}\alf_{j+2}\cdots \alf_{l-1})A.\]
Replacing $b_l$ by $\mathcal{B}(\alf_{l+1})/2\pi$, and that $\alf_i\alf_{i+1}\leq 1/2$, for all $i$, we come up with 
\[b_j\leq \dfrac{1}{2\pi}\alf_{j+1}\alf_{j+2}\cdots \alf_{l}\mathcal{B}(\alf_{l+1})+ 4A.\]

As $l$ tends to infinity, the first term in the above sum, which is the tail of the Brjuno series for $\alf_{j+1}$, 
tends to zero.
\end{proof}

\begin{lem}\label{contracting-area}
There exists a constant $\eps>0$ such that if $\alf_j<\eps$ for some $j\geq m$, then 
$\area B_{j-1}<\frac{1}{2}\area B_j$.
\end{lem}

\begin{proof}
Using Proposition~\ref{main-estimate} with $h=f_j$, $\chi_j:= \chi_{f_j}$, and $r=1/2$,  
\begin{align*}
\area B_{j-1}&= \area \chi_j(\Pi_j)\\
&\leq \int_{\Pi_j} |\chi_j'|^2\, \frac{dz\wedge d\bar{z}}{2}\\
&\leq (\frac{1}{\alf_j}+\Bk'')\cdot \max_{\ell\in \mathbb{Z}} 
\int_{\{w\in \Pi_j| \ell \leq \Re w\leq \ell+1]\} } |\chi_j'|^2\, \frac{dz\wedge d\bar{z}}{2}\\
&\leq (\frac{1}{\alf_j}+\Bk'')(2 C\alf_n e^{-2\pi D})^2 \area B_j\\
&\leq \frac{1}{2}\area B_j,
\end{align*}
with the last inequality for small enough values of $\alf_j$.
\end{proof}

\begin{propo}
The  set $\Gamma$  has zero area.
\end{propo}
\begin{proof}
As $\Gamma=\cup_{m\geq 0} \Gamma^m$, it is enough to show that each $\Gamma^m$ has zero area.
Clearly, $\Gamma^m$ has zero area if and only if $\Gamma_m$ has zero area. 
To prove this, we show that $B_m$ has zero area in the two separate cases. 
\medskip

{\em Case i:} Eventually $\alf_j\geq \eps$, where $\eps$ is obtained in Lemma~\ref{contracting-area}.

\noindent For these values of $\alf$, $\mathcal{B}(\alf_j)$ is 
uniformly bounded, and hence, the set $\Gamma$ is empty for large enough $D$.   
(In this case, $\alf$ is of bounded type, and the boundary of the Siegel disk is known to be a quasi-circle 
passing through the critical point.) 
\medskip

{\em Case ii:} There are arbitrarily large $j$ with $\alf_j<\eps$.

\noindent By Lemma~\ref{bounded-height}, $\area B_j\leq 4 A$, for all $j$. 
On the other hand, there are infinitely many levels $j$ with $\area B_{j-1}< \frac{1}{2}\area B_j$.  
\end{proof}
\subsection{Area of $\Lambda$}\label{SS:porous-part}
\begin{propo}\label{prop:non-uniformly-porous}
The set $\cap_{n=0}^\infty \Omega^n_0\cap J(P_\alf)$ is non-uniformly porous at every point in $\Lambda$. 
In particular, $\Lambda$ has zero area.
\end{propo}

\begin{rem}
In \cite{Ch10-I} we prove that when $\alpha$ is a non-Brjuno number, for every non-zero 
point in $\cap_{n=0}^\infty \Omega^n_0$, there exists $D$ such that infinitely often 
$\Im \zeta_n \leq D/\alpha_n$. 
In other words, the intersection is contained in the union of all $\Lambda$, over all $D\in (0, +\infty)$.  
Then, this was used to show that such a point can not be the Lebesgue density point of $\pc(f_0)$. 
For the reader's convenience we present an argument, similar in principle but 
simpler in technical details, to produce the same result.
\end{rem}

\begin{lem}\label{L:free-balls}
There are $\delta_2, \delta_3\in \mathbb{R}$ such that for every $n\geq 0$ with $\Im \zeta_n< D/\alpha_n$, 
there exist a simply connected domain $V_n\subset \cc\setminus \{0\}$ which equipped with its conformal 
metric $\rho_n|dz|$ of constant curvature $-1$ satisfies the following: 
\smallskip

\begin{itemize}
  \setlength{\itemsep}{.4em}
\item[i.] $V_n\subset  \Dom f_n\cap f_n(\Dom f_n)$, and $z_n\in V_n$;
\item[ii.] There exists $z_n'\in V_n$ with $d_{\rho_n}(z_n, z_n')\leq \delta_2$ such that 
\[B_{\rho_n}(z'_n, \delta_3) \cap \Omega^0_n=\varnothing, \text{ and } 
f_n(B_{\rho_n}(z'_n, \delta_3)) \cap \Omega^0_n=\varnothing;\]
\item[iii.] $\sup_{z,w\in V_n} |\arg z-\arg w|\leq \Bk'$, for every branch of argument defined on $V_n$.
\end{itemize}
\end{lem}

\begin{proof}
Recall that $\zeta_n$ belongs to the set $\cup_{j=1}^{\Bk'}B(j, \delta_1)\cup W_{f_n}$, see 
Equation \eqref{E:W}. 
Since the sectors $\C_n^{-j}\cup (\Csh_n)^{-j}$, for $ j=0,1,2,\dots, k_n$, 
are compactly contained in $\Dom f_n$, $k_n$ is uniformly bounded independent of $n$ by Lemma~\ref{turning}, 
the continuous dependence of the Fatou coordinate on the map, and the compactness of the class $\ff$, 
there exists a constant $\eps_0$ such that $\Phi_n^\dagger$ is defined on $B(W_{f_n}, \eps_0)$ 
with 
\[\Phi_n^\dagger(B(W_{f_n}, \eps_0))\subset \Dom f_n.\] 

By Theorem~\ref{Ino-Shi2}, the map $f_n: \Dom f_n\to \cc$ extends onto a domain $V'$ with the modulus of 
$V'\setminus \Dom f_n$ uniformly bounded from below, independent of $n$. 
By the Koebe distortion Theorem, this implies that $\Dom f_n$, for all $n$, is contained in a fixed ball about $0$. 
Similarly, all of $f_n$, for $n\geq 1$, are univalent on the ball $B(0, 1/12)$.
By the definition of renormalization, this implies that there exists a constant $M_0>2$ such that 
each $\Phi_n^\dagger$, for $n\geq 1$, is univalent above the line $ \Im w= M_0$, and that 
$W_{f_n}$ lies above the line $\Im w =-M_0$. 

Let $U\subset \cc$ be a domain with the hyperbolic metric $\rho |dz|$ of constant curvature equal to $-1$.  
Using the 1/4-Theorem, one can show the following estimate on $\rho$ at a given point $z\in U$
\begin{equation}\label{E:hyperbolic-metric}
\tfrac{1}{2 d(z, \partial U)} \leq \rho(z)\leq \tfrac{2}{d(z, \partial U)}.
\end{equation}

The proof of the lemma is broken into two situations based on the position of $\zeta_n$.
\medskip

{\em The proof when $\Im \zeta_n< M+1$:} 
\medskip

If $\Re \zeta_n\in (7/8, 1/\alpha_n-\Bk-1)$ define 
\[r_n':= \tfrac{1}{2} \min \{\Re \zeta_n, \tfrac{1}{\alpha_n}-\Bk-\Re \zeta_n\},\;\zeta_n':=\Re\zeta_n-M\Bi-r_n' \Bi.\]
Let $V_n:=\p_n$ and $z_n':= \Phi_n^{-1}(\zeta_n')$. 
The ball $B(\zeta_n', r_n')$ and its translation by one are below the heights $-2$ and $-M$. 
Hence, by the the choice of $M$, $\Phi_n^{-1}(B(\zeta_n',r_n'))$ and $f_n(\Phi_n^{-1}(B(\zeta_n',r_n')))$ are 
outside of $\Omega_n^0$.

On the other hand, one can easily verify that the hyperbolic distance within $\Phi_n(V_n)$ from $\zeta_n$ to 
$\zeta_n'$ is uniformly bounded from above, and that $B(\zeta_n', r_n')$ contains a hyperbolic ball of definite 
size independent of $n$.
As $\Phi_n$ is an isometry from $V_n$ to $ \Phi_n(V_n)$, with respect to the hyperbolic metrics, we conclude 
the second part of the lemma. 
The other parts are obvious here.

If $\Re \zeta_n\geq 1/\alpha_n-\Bk-1$, define 
\begin{align*}
V_n'&:=B(W_{f_n}, \eps_0)\cup B(W_{f_n}+1, \eps_0)
\cup \{\zeta\in \cc ; \Re \zeta\in (\tfrac{1}{\alpha_n}-\Bk-2, \tfrac{1}{\alpha_n}-\Bk)\},\\
\zeta_n'&:=-2M \Bi+ \tfrac{1}{\alpha_n}-\Bk-1.
\end{align*}
Let $V_n:= \Phi_n^\dagger(V_n')$ and $z_n':=\Phi_n^\dagger (\zeta_n')$. 
By the contraction of $\Phi_n^\dagger$ with respect to hyperbolic metrics, $d_{\rho_n}(z_n, z_n')$ is 
uniformly bounded from above.  
By the Koebe distortion Theorem $\Phi_n^{-1}(B(\zeta_n',1))$ contains a hyperbolic ball of definite size about
$z_n'$.
The other properties stated in the lemma follow similarly. 

When $\alpha_n$ is bigger than some constant $\eps$, $\Im \zeta_n < D/\eps$ and the problem is 
reduced to the above case. 
Thus, it remains to prove the lemma for small values of $\alpha_n$.
\medskip

{\em The proof when $\alpha_n$ is small:}
\medskip

If $\Re \zeta_n\in (7/8, 1/\alpha_n-\Bk-1)$ and $ \Im \zeta_n>M+1$ define 
\begin{align*}
r_n'&:= \tfrac{1}{4} \min \{\Re \zeta_n, \tfrac{1}{\alpha_n}-\Bk-\Re \zeta_n\},\\
\zeta_n'&:=\Re\zeta_n-M\Bi-r_n' \Bi,\\
V_n'&:=\{\zeta\in \cc ;  |\Re (\zeta - \zeta_n)|<3r_n'\}\cup 
\{\zeta\in \cc ; \Im \zeta>M, |\Re (\zeta- \zeta_n)|< \tfrac{1}{2\alpha_n}\}.  
\end{align*}
Using the estimate in \eqref{E:hyperbolic-metric}, one can verify that the hyperbolic distance within $ V_n'$ 
from $\zeta_n$ to $\zeta_n'$ is uniformly bounded independent of $n$.
Similarly, $B(\zeta_n', r_n')$ contains a hyperbolic ball of definite size independent of $n$ about $\zeta_n'$.  

The map $\Phi_n^{-1}$ extends to a map defined on $V_n'$, using the dynamics of $ f_n$. 
Furthermore, it is univalent for sufficiently small values of $\alpha_n$. 
Now, define $V_n:=\Phi_n^{-1}(V_n')$ and $z_n':=\Phi_n^{-1}(\zeta_n')$. 
The isometry $\Phi_n^{-1}$ preserves the respective distance on $V_n'$ and $V_n$, hence 
$d_{\rho_n}(z_n, z_n')$ is uniformly bounded independent of $n$, 
and $\Phi_n^{-1}(B(\zeta_n',r_n' ))$ contains a hyperbolic ball of definite size about $ z_n'$.
The other properties in the lemma should be clear as well. 

If $\zeta_n\in W_n$ and $\Im \zeta_n\geq M+1$ define 
\begin{align*}
r_n'&:=(\Im \zeta_n-M)/(4D),\\
\zeta_n'&:=(\tfrac{1}{\alpha_n}-\Bk-r_n')-\Bi(M+r_n'),\\
V_n'&:=\{\zeta\in \cc ; \tfrac{1}{2\alpha_n}<\Re \zeta <\tfrac{1}{\alpha_n}-\Bk\} \cup 
\{\zeta\in \cc ; \Im \zeta>M, \tfrac{1}{\alpha_n}-\Bk\leq \Re \zeta < \tfrac{5}{4\alpha_n}.\}  
\end{align*}
Similarly, extend $\Phi_n^\dagger$ to a univalent map on $V_n'$, using the dynamics of $f_n$, and 
let $V_n:=\Phi_n^\dagger(V_n')$ and $z_n':=\Phi_n^\dagger(\zeta_n')$. 
By our choice of  $ M$, $\Phi_n^\dagger$ is an isometry from $V_n'$ to $V_n$ with respect to the hyperbolic metrics.
Hence, one can verify that $d_{\rho_n}(z_n, z_n')$ is uniformly bounded from above independent of $n$, 
and, $\Phi_n^{-1}(B(\zeta_n',r_n' ))$ contains a hyperbolic ball of definite size about $ z_n'$.
\end{proof}

Fix an $n$ with $\Im \zeta_n<D/\alpha_n$. 
We inductively define a sequence of domains $V_n, V_{n-1}, V_{n-2}, \dots, V_0$, and maps 
$g_n, g_{n-1}, \dots, g_1$, where each $V_i=\Omega_i^0\setminus I_{i,r(i)}$, for some 
$r(i)\in [3/2, 1/\alpha_i-\Bk]$, $z_i\in V_i$, and each $g_i: V_i\to V_{i-1}$ is a holomorphic or 
anti-holomorphic map with $g_i(z_i)=z_{i-1}$. 

Let $V_n$ be the domain obtained in Lemma~\ref{L:free-balls}. 
Assuming $V_l$ is defined for some $l$, the map $g_l$ and the domain $V_{l-1}$ are defined according to one 
of the following scenarios.
Since $V_l$ is a simply connected domain which spirals at most $\Bk'$ times about $0$, 
there is an inverse branch of $\ex$ defined on $V_l$, denoted by $\eta_l$, that maps $z_l$ to $\zeta_{l-1}$.  
As $\eta_l(V_l)$ contains $\zeta_l$, by \eqref{E:place-of-zeta}, one of the following must occur:
\medskip
\begin{itemize}
\item[(a)] $\Re \eta_{l}(V_l) \subset (-\Bk'+7/8, \Bk'+1/2)$
\item[(b)] $\Re \eta_{l}(V_l)\subset (1/2, 1/\alpha_l-\Bk-1/2)$
\item[(c)] $\Re \eta_{l}(V_l)\subset (1/\alpha_l-\Bk-1/2, 1/\alpha_l-\Bk+k_n+2\Bk')$
\end{itemize}
\medskip

When (a) occurs, we must have $\zeta_{l-1}\in \cup_{j=1}^{\Bk'} B(j, \delta_1)$, and hence by 
Lemma~\ref{L:extra-space}, 
$\eta_l(V_l)$ contains one of these balls. 
This implies that $\eta_l(V_l)$ avoids $\delta_1$-neighborhoods of the points $0,-1,-2, \dots, -\Bk'+1$. 
Now, we define $g_l:=\Phi_l^{-1}\circ \eta_l$ on $V_l$.    
By definition, $g_l(z_l)=z_{l-1}$.  
By our assumption \eqref{E:k'-condition}, $g_l(V_l)$ avoids the curve $I_{l-1, \Bk'+1}$, and therefore, 
$g_l$ maps $V_l$ into the set $V_{l-1}:=\Omega_{l-1}\setminus I_{l-1, \Bk'+1}$.  
Indeed, by Lemma~\ref{L:extra-space}, 
\begin{equation}\label{E:buffer-when-a}
\Phi_l^{-1}(B(\eta_l(V_l), \delta_1))\subseteq V_{l-1}. 
\end{equation}
\medskip

When (b) occurs, let $g_l:=\Phi_l^{-1}\circ \eta_l$, and 
$V_{l-1}:=\Omega_{l-1}\setminus I_{l-1, \alpha_{l-1}^{-1}-\Bk}$. 
Here also, $g_l$ maps $z_l$ to $z_{l-1}$, and property \eqref{E:buffer-when-a} holds.  
\medskip
  
When (c) occurs, let $g_l:=\Phi_l^\dagger \circ \eta_l$, and 
$V_{l-1}:=\Omega_{l-1}\setminus I_{l-1, \alpha_{l-1}^{-1}-\Bk-1}$. 
Here, we have 
\begin{equation}\label{E:buffer-when-c}
\Phi_l^\dagger(B(\eta_l(V_l), \delta_1))\subseteq V_{l-1}. 
\end{equation}
and $g_l(z_l)= z_{l-1}$, here as well. 
\medskip

The maps $g_i$ enjoy the following properties. 
\begin{lem}\label{L:nice-chain}
There are constants $\delta_4, \delta_5$ such that for every $i=1,2 \dots, n$ 
\begin{itemize}
\item[i.] the map $g_i:(V_i, \rho_i)\to (V_{i-1}, \rho_{i-1})$ is uniformly contracting, that is, for every $z\in V_i$ 
\[ \rho_{i-1}(g_i(z)) |g_i'(z)|\leq \delta_4\rho_i(z);\] 
\item[ii.] the map $g_i:V_i \to V_{i-1}$ is univalent on the hyperbolic ball
\[\{z\in V_i\mid d_{\rho_i}(z, z_i)<\delta_5\};\]
\item[iii.] Each map $g_i$ is either univalent or has a unique simple critical point in $V_i$.
\end{itemize}
\end{lem}
\begin{proof}
{\em Part i.}
By the definition of $g_i$, and Equations \eqref{E:buffer-when-a} and \ref{E:buffer-when-c}, each $g_i$ can 
be decomposed as 
\begin{displaymath}
\xymatrix{
(V_i,\rho_i) \ar[r]^-{\eta_i} &\; (\eta_i(V_i), \tilde{\rho}_i) \ar@{^{(}->}[r]^-{inc.} &\;  
(B_{\delta_1}(\eta_i(V_i)),\hat{\rho}_i)\;\; \ar[r]^-{\Phi_l^{-1} \text {or } \Phi_i^\dagger}\; & \;(V_{i-1},\rho_{i-1}),}
\end{displaymath}
where $\tilde{\rho}_i(z)|d z|$ and $\hat{\rho}_i(z)|dz|$ denote the hyperbolic metric of constant $-1$ curvature on 
$\eta_i(V_i)$ and $B_{\delta_1}(\eta_i(V_i))$, respectively. 

By the Schwartz-Pick Lemma, the first map and the last map in the above decomposition are non-expanding, so it is 
enough to show that the inclusion map is uniformly contracting.
To this end, fix an arbitrary point 
$\zeta_0$ in $\eta_i(V_i)$, and define 
\[H(\zeta):=\zeta +(\zeta-\zeta_0)\frac{\delta_1}{(\zeta-\zeta_0+2\Bk'+1)}: \eta_i(V_i)\ra \cc.\]
As $\diam \Re(\eta_i(V_i))\leq \Bk'$, for every $\zeta \in \eta_i(V_i)$ we have $|\Re(\zeta-\zeta_0))|\leq \Bk'$ . 
Consequently, $|\zeta-\zeta_0|<|\zeta-\zeta_0+2\Bk'+1|$, and then 
\begin{align*}
|H(\zeta)-\zeta|&= \delta_1 |\frac{\zeta-\zeta_0}{\zeta-\zeta_0+2\Bk'+1}| < \delta_1.
\end{align*}
This implies that, $H(\zeta)$ is a holomorphic map from $\eta_i(V_i)$ into  $B_{\delta_3}(\eta_i(V_i))$. 
By Schwartz-Pick Lemma, $H$ is non-expanding. 
In particular, at $H(\zeta_0)=\zeta_0$ we have
\[\tilde{\rho}_i(\zeta_0) |H'(\zeta_0)|= \tilde{\rho}_i(\zeta_0)(1+\frac{\delta_1}{2\Bk'+1}) \leq \hat{\rho}_i(\zeta_0).\]
That is, $\tilde{\rho}_i(\zeta_0)\leq \delta_4\cdot \hat{\rho}_i(\zeta_0)$
with $\delta_4=(2\Bk'+1)/(2\Bk'+1+\delta_1)$.

\medskip
{\em Part ii.}
In cases (a) and (b) of the definition of the map $g_i$ both $\eta_i$ and $\Phi_i^{-1}$ are univalent, hence $g_i$ 
is univalent on the whole domain $V_i$. 

In case (c) of the definition of the map $g_i$, $\eta_i$ is univalent, and $\Phi_{i-1}^\dagger$ has $\Bk'$ simple 
critical points that are mapped to 
\[f_{i-1}(\cp), f_{i-1}\co{2}(\cp), \dots, f_{i-1}\co{\Bk'}(\cp).\]
However, in this case
\[z_{i-1}\notin \Phi_{i-1}^{-1}(B(j,\delta_1)), \text{ for } j=1,2,\dots, \Bk'.\] 
The curve, $\Phi_{i-1}^{-1}(-2 \Bi+t)$, for $t\in [0,\Bk']$, belongs to the boundary of $V_{i-1}$.
By the continues dependence of the Fatou coordinate on the map, there exists $\delta_5>0$ such that 
\[B_{\rho_{i-1}}(f_{i-1}\co{j}(\cp),\delta_5)\subseteq \Phi_{i-1}^{-1}(B(j, \delta_1)), \text{ for } j=1,2,\dots, \Bk'.\]
Hence, combining with the above argument, the branched covering $\Phi_{i-1}^\dagger$ has no critical 
value in $B_{\rho_{i-1}}(z_{i-1},\delta_5)$.
Let $W$ be a connected component of $g_i^{-1}(B_{\rho_{i-1}}(z_{i-1}, \delta_5))$ containing $z_i$.
As $B_{\rho_{i-1}}(z_{i-1}, \delta_5)$ is simply connected, $g_i$ is univalent on $W$. 
In particular, by the contraction proved in Part i, $g_i$ is univalent on $B_{\rho_i}(z_i, \delta_5)\subseteq W$.  
\medskip

{\em Part iii.}
In both cases (a) and (b), the map $g_i$ is a composition of the univalent map $\eta_i$ and an inverse branch of
$\Phi_i$ on $\eta_i(V_i)$. 
Hence, it is univalent. 

In case (c), the lift $\eta_i(V_i)$ cannot contain any two points apart by an integer value. 
On the other hand, by the equivariance property of the Fatou coordinates, $\Phi_{i-1}^\dagger$ has 
$\Bk'$ number of simple critical points apart by one. 
Therefore, no two of these critical points can belong to $\eta_i(V_i)$. 
\end{proof}

The composition
\[G_n:= g_1\circ \dots \circ g_{n-1}\circ g_n:V_n \to V_0,\]
maps $z_n$ to $z_0$.
\begin{lem}\label{safe-jump} 
There are constants $D_1$ and $\delta_6\in (0,1)$ such that for every $n$ with 
$\Im \zeta_n< D/\alpha_n$, there is $r_n\in(0,1)$ such that 
\medskip
 
\begin{itemize}
  \setlength{\itemsep}{.8em}  
\item[i.] $G_n(B_{\rho_n}(z_n', \delta_3)) \cap \pc(f_0)=\varnothing;$
\item[ii.] $B(G_n(z_n'), r_n)\subset G_n(B_{\rho_n}(z_n', \delta_3))$, and  
$|G_n(z_n')- z_0| \leq D_3\cdot r_n;$
\item[iii.] $r_n\to 0$, as $n\to \infty$.
\end{itemize}
\end{lem}
\begin{proof}
{\em Part i.}
First note that the conformal change of coordinate $\Psi_n$ preserves the post-critical sets. That is, 
\begin{equation*}\label{E:pc-invariance}
\Psi_n(\pc(f_n)\cap \p_n)=\pc(f_0)\cap \Psi_n(\p_n).
\end{equation*}
The next step is to see that there exists an integer $l\leq q_n$ such that 
\begin{equation*}\label{E:lift-then-iterate}
G_n(B_{\rho_n}(z_n', \delta_3))=f_0\co{l}(\Psi_n(B_{\rho_n}(z_n', \delta_3))).
\end{equation*}
The existence of such an $l$ follows from the definition of the renormalization. 
It is less than $q_n$ since when defining each $g_i$, $i=n-1, n-2, \dots 1$, the map $f_i$ in iterated 
less than $1/\alpha_i$.

Now, 
\begin{align*}
&  && \;\; \qquad G_n(B_{\rho_n}(z_n', \delta_3)) \cap \pc(f_0)\neq \varnothing \\
&(\text{As } f_0(\pc(f_0))= \pc(f_0))  
&& \Rightarrow \quad  f_0\co{q_n-l}\circ G_n(B_{\rho_n}(z_n', \delta_3))\cap \pc(f_0)\neq \varnothing.\\
&
&& \Rightarrow \quad f_0\co{q_n}\circ \Psi_n(B_{\rho_n}(z_n', \delta_3))\cap \pc(f_0)\neq \varnothing \\
&(\text{By Lemma~\ref{conjugacy}})  
&& \Rightarrow \quad \Psi_n\circ f_n (B_{\rho_n}(z_n', \delta_3))\cap \pc(f_0)\neq \varnothing \\
&
&& \Rightarrow \quad f_n (B_{\rho_n}(z_n', \delta_3))\cap \pc(f_n)\neq \varnothing. 
\end{align*}
The last line contradicts Lemma~\ref{L:free-balls}-ii, since Proposition~\ref{neighbor} implies that 
$\pc(f_n)$ is contained in $\Omega_n^0$.
\footnote{Indeed, one can show that $G_n(B_{\rho_n}(z_n', \delta_3)) \cap \Omega_0^n=\varnothing$, 
see \cite{Ch10-I}-Lemma 4.7.}
\medskip

{\em Part ii.}
Let $m$ be the smallest positive integer with $(\delta_3+\delta_2)\cdot (\delta_4)^m\leq \delta_5$, 
and then decompose the map $G_n$ into its first $m+1$ iterates and the rest as 
\[G_n:=G_n^u\circ G_n^{l}, \text{ where } 
G_n^l:= g_{n-m}\circ \cdots \circ g_{n-1}\circ g_n, \text{ and }
G_n^u:=g_1\circ g_2\circ \cdots \circ g_{n-m-1}.\]

Let $\gamma_n$ denote the hyperbolic geodesic in $V_n$ connecting $z_n$ to $z_n'$. 
By the above choice of $m$ and Lemma~\ref{L:nice-chain}-i, we have 
\[G_n^l(B_{\rho_n}(z_n', \delta_3)\cup \gamma_n)\subset B_{\rho_{n-m}}(z_{n-m}, \delta_4\delta_5).\] 
Since each $g_i$, for $i=n-m-1, n-m-2, \dots, 1$, is contacting and univalent on the ball $B_{\rho_i}(z_i, \delta_5)$, 
$G_n^u$ is univalent on $ B_{\rho_{n-m-1}}(z_{n-m-1}, \delta_5)$.  
By the Koebe distortion Theorem, $G_n^u$ has uniformly bounded, independent of $n$, distortion on  
$G_n^l(B_{\rho_n}(z_n', \delta_3)\cup \gamma_n)$. 

On the other hand, $G_n^l$ is a uniformly bounded number of iterates of some $g_i$ that belong to a compact 
class of maps in the compact open topology. 
More precisely, by definition, each $g_i$ is a composition of $\eta_i$ and either of $\Phi_i^{-1}$ or $\Phi_i^\dagger$.
The map $\eta_i$ is univalent and, by Koebe's Theorem, has bounded distortion on hyperbolic sets of bounded diameter, 
and each of the other maps has extension over a larger domain by Equation \eqref{E:buffer-when-a}.    

Since near $z_0$ the hyperbolic and the Euclidean metrics are equivalent, 
combining the above arguments, one concludes that $G_n(B_{\rho_n}(z'_n, \delta_3))$ contains an round ball of 
Euclidean size proportional to its Euclidean diameter, and that the Euclidean diameter of  
$G_n(B_{\rho_n}(z'_n, \delta_3))$ is proportional to $|G_n(z_n')-z_0|$. 
\medskip

{\em Part iii.}
The hyperbolic diameter of $G_n(B_{\rho_n}(z'_n, \delta_3))$ tends to zero exponentially fast, by 
Lemma~\ref{L:nice-chain}-i. 
Hence, by equivalence of the hyperbolic and Euclidean metrics at local scales, $r_n$ must tend to zero.
\end{proof}

\begin{proof}[Proof  of  Proposition~\ref{prop:non-uniformly-porous}] 
Given $z_0\in \Lambda$,  we have an increasing sequence $n_i$ with $\Im \zeta_{n_i}< D/\alpha_{n_i}$.
Lemma~\ref{L:free-balls} provides us with the balls $B_{\rho_{n_i}}(z_n',\delta_3)$ in the complement of 
 $\pc(f_{n_i})$. 
The maps $\G_{n_i}$, by Lemma~\ref{safe-jump}, carry these balls to domains in the complement of $\pc(f_0)$ 
containing round balls $B(G_n(z_{n_i}'),r_{n_i})$ at arbitrarily small scales near $z_0$. 
Hence, $z_0$ can not be a Lebesgue density point of $\Lambda$. 
\end{proof}
\subsection{Non-uniform porosity of decorations} 
For Brjuno values of $\alf$, let $\Delta_\alf$ denote the Siegel disk of $P_\alf$, and $\overline{\Delta}_\alf$ 
denote its closure. 
Here we show that $\pc(P_\alf)$ is non uniformly porous at every point in $\pc(P_\alf)\setminus\overline{\Delta}_\alf$. 

\begin{lem}\label{decorations}
We have $\pc(P_\alf) \setminus \overline{\Delta}_\alf\subset \Lambda$.
\end{lem}
\begin{proof}
Let $z\in  \pc(P_\alf) \setminus \overline{\Delta}_\alf$ and assume on the contrary that there is 
an integer $N$ with $\Im \zeta_m\geq D/\alf_m$, for all $m\geq N$. 
 
By the first argument in the proof of Lemma~\ref{bounded-height}, we have 
\[d(\zeta_j, \Phi_j(\partial \Delta_j))\leq \mathcal{B}(\alf_{j+1})/2\pi,\] 
for all $j$.
The boundary of the Siegel disk of a map is sent to the boundary of the Siegel disk of its 
renormalization under the change of coordinates. 
Now it follows from Lemmas~\ref{recursive-heights} and \ref{bounded-height} that for all $j\geq N$ 
we have $d(\zeta_j, \Phi_j(\partial \Delta_j))\leq 4A$. 
By a contraction argument for the changes of coordinates from a level of renormalization to the previous 
level with respect to appropriate hyperbolic metrics, see Lemma~\ref{L:nice-chain}-i, 
one can see that all these distances must be zero. 
This implies that $z \in \overline{\Delta}_\alf$. 
\end{proof}

Theorem~\ref{thm:porosity-of-decorations} follows from the above lemma 
and Proposition~\ref{prop:non-uniformly-porous}.
\section{Accumulation on the critical point}
Let $f:\RS\to \RS$ be a rational mapping with $J(f)\neq \RS$, and $U$ be an arbitrary neighborhood of its 
post-critical set $\pc(f)$. 
By a lemma of \cite{Ly83b}, the orbit of Lebesgue almost every point in $J(f)$ eventually stays in $U$. 
Therefore, by virtue of Proposition~\ref{neighbor}-ii, for every $n\geq 0$ and almost every $z\in J(P_\alf)$,
there is a positive integer $k$ such that the orbit of $f\co{k}(z)$ is contained in $\Omega^n$.
(Indeed, it is proved in \cite{Ch10-I} that while the orbit of a point stays in some $\Omega^n$, it visits all the 
sectors involved in the union.)  

Here we introduce a shrinking sequence of sets containing $\cp$, denoted by $Q^n_0$, that are visited by the orbit 
of almost every $z\in J(P_{\alf})$. 

Recall the sets $\C_n^{-k_n}$ and $(\Csh_n)^{-k_n}$ obtained in the definition of $\rr (f_n)$. 
For every $n\geq 0$, define the sets
\begin{gather*}
W_n:= \C_n^{-k_n}\subset S_n^0, \text{ and } W^n:= \Psi_n(W_n)\subset S_0^n
\end{gather*}
Note that for every $n\geq 0$ there is a pre-critical point of $f_n$ in $W_n$ and hence, 
by Lemma~\ref{conjugacy}-(2), there is a pre-critical point of $f_0$ in $W^n$. 
For every $n\geq 0$, define 
\[Q^n:=f_0\co{\tau(n)}(W^n),\] 
where $\tau(n)$ is the smallest positive integer with  $\cp\in \daron (f_0\co{\tau(n)}(W^n))$.

\begin{lem}\label{good-preimage}
For all $n\geq 0$ and all $z \in (\Csh_n)^{-k_n}$, we have $\Im \phi_n(z)\geq -2$.
\end{lem}
\begin{proof}
By the definition of the renormalization $\rr f_n$, it is enough to show that 
\[\forall w\in B(0,\frac{4}{27}e^{-4\pi}),\text{ we have } f_{n+1}^{-1}(w)\in B(0,\frac{4}{27}e^{4\pi}),\]
where $f_{n+1}^{-1}(w)$ is the unique pre-image of $w$.
The map $f_{n+1}$ has a unique simple critical point mapped to $-4/27$. 
Hence, there is an inverse branch of $f_{n+1}$ defined on the ball $B(0, -4/27)$. 
By Koebe distortion Theorem, it follows that $f_{n+1}^{-1}(B(0,\frac{4}{27} e^{-4\pi}))$ is contained in 
$B( 0, \frac{4}{27}e^{4\pi})$.
\end{proof}

\begin{lem}\label{steps}
Assume that $z\in \Omega^n\setminus \Omega^{n+1}$ and $f_0(z)\in \Omega^{n+1}$, for some $n\geq 0$.
Then, either $z\in Q^n$ or $f_0(z)\in W^{n+1}$ (hence, $f_0\co{(\tau(n+1)+1)}(z)\in Q^{n+1}$).
\end{lem}

\begin{proof}
There is a 
\[j\in \{0,1,\dots, q_{n+1}(k_{n+1}+\lfloor 1/\alf_{n+1}\rfloor-\Bk-1)+q_n\},\] 
such that 
$f_0(z)\in f_0\co{j}(S_0^{n+1})$, by the definition of $\Omega^{n+1}$.
If $j\neq 0$, there is a $w\in f_0^{j-1}(S^{n+1}_0)$ with $f_0(w)=f_0(z)$. 
If $z\notin Q^n$, as $f_0:\Omega^n\setminus Q^n\to \cc$ is univalent, we must have $z=w\in\Omega^{n+1}$ 
which contradicts the assumption in the lemma. 
Therefore, either $j=0$ or $z\in Q^n$.

If $j=0$, then 
\begin{align*}
f_0(z) \in S^{n+1}_0 =&  \Psi_{n+1}(S_{n+1}^0)\\
                                    =&  W^{n+1}\cup \Psi_{n+1}((\Csh_{n+1})^{-k_{n+1}})
\end{align*}                                  
However, if $f_0(z)\in \Psi_{n+1}((\Csh_{n+1})^{-k_{n+1}})$, by Lemma~\ref{good-preimage}, $z\in \Omega^{n+1}$. 
This finishes the proof of the lemma.
\end{proof}

Let us define the \textit{eccentricity} of a domain $Q$ at a point $q\in \daron Q$ as 
\[\ecc(Q,q):= \frac{\inf \{r\mid Q \subset B(q, r)\}}{\sup\{r\mid B(q, r)\subset Q\}}.\]

\begin{lem}\label{nice&round}
We have 
\begin{itemize}
\item $\sup_n \ecc(Q^n, \cp)<\infty$,
\item $\lim_{n\to \infty} \diam Q^n=0$.  
\end{itemize}
\end{lem}

\begin{proof}
The proof of the first statement is broken into rour steps.
\medskip

{\em Step 1.} For every $n\geq 0$, there exists a simply connected domain $\widehat{W}_n$
containing $W_n$ with the following properties.
\begin{itemize}
\item $\widehat{W}_n\subset \p_n$, $f_n\co{k_n}(\widehat{W}_n)\subset \p_n$,
\item $\inf_n \!\!\!\mod (\widehat{W}_n\setminus W_n) > 0$,
\item $f_n\co{k_n}:\widehat{W}_n\to f_n^{k_n}(\widehat{W}_n)$  is a proper branched covering of 
degree two. 
\end{itemize}
Recall that $f_n\co{k_n}: W_n\to \C_n$ is a proper branched covering of degree two, 
$k_n$ is uniformly bounded (Lemma~\ref{turning}), $B_{1/2}(W_n)\subset \p_n$, and 
$B_{1/2}(\C_n)\subset \p_n$.
Now the above step follows from the compactness of the class $\ff$, and the continuous 
dependence of $W_n$ on $f_n$.
\medskip

{\em Step 2.} If $f_n^{-k_n}(\cv_{f_n})$ is the unique preimage of $\cv_{f_n}$ contained in $W_n$
then,   
\[\sup_n \ecc (W_n, f_n^{-k_n}(\cv_{f_n}))<\infty.\]

By Koebe distortion Theorem, $\sup_n \ecc (\C_n, \cv_{f_n}) <\infty$. 
Indeed, $\Phi_n^{-1}$ is defined and univalent on the $1/2$-neighborhood of $\Phi_n(\C_n)$.
The map $f_n\co{k_n}: W_n\to \C_n$ belongs to a compact class by Step 1.  
This implies the above statement on eccentricity of $W_n$. 
 \medskip
  
{\em Step3.} $f_0\co{\tau(n)}$ is univalent on $\Psi_n(\widehat{W}_n)$, and 
$f_0\co{\tau(n)}(\Psi_n( f_n^{-k_n}(\cv_{f_n})))=\cp.$

By Step 1 and Lemma~\ref{conjugacy}, the map 
$f_0\co{(k_nq_n+q_{n-1})}: \Psi_n(\widehat{W}_n)\to \p_0$ is a proper branched covering of 
degree two mapping $\Psi_n(f_n^{-k_n}(\cv_{f_n}))$ to $\cv$.
We must have 
\begin{equation}\label{bound-tau}
\tau(n)<k_nq_n+q_{n-1}.
\end{equation} 
On the other hand, $f_0\co{\tau(n)}(W^n)$ contains a critical point, and therefore,  
$f_0\co{\tau(n)}$ must be univalent on $\Psi_n(\widehat{W}_n)$.
\medskip

{\em Step 4. } We have $\sup_n \ecc(Q^n, \cp)<\infty$. 

By definition $Q^n:=f_0\co{\tau(n)}\circ \Psi_n(W^n_0)$, and by Step 3, $f_0\co{\tau(n)}\circ \Psi_n$ 
has univalent extension onto $\widehat{W}_n$ with 
$\inf_n\!\!\!\mod (\widehat{W}_n\setminus W_n)>0$.  
Again by the Koebe distortion Theorem, there exists a constant $C(\eps)$ such that 
\[\ecc (Q^n, \cp)\leq C(\eps) \ecc(W_n,f_n^{-k_n}(\cv_{f_n})).\]
Now Step 4 follows from Step 2.
\medskip

To prove the second part, we use Montel's Normal Family Theorem. 
Since $Q^n\subset f_0\co{\tau(n)} (S_0^n)$, by definition of $\Omega^n$,
$f_0$ can be iterated at least 
\[q_n(k_n+\lfloor 1/\alf_n\rfloor-\Bk-1)+q_{n-1}-\tau(n)\] 
times on $Q^n$ with values in $\Omega^n$.
By \eqref{bound-tau}, the above quantity goes to infinity as $n$ goes to infinity. 

Now if $\limsup_{n\to \infty} \diam( Q^n)\neq 0$, by the first part, a ball of positive size about $\cp$ 
is contained in infinitely many $Q^n$'s.
By Montel's Normal Family Theorem, this implies that the critical point belongs to the 
Fatou set of $f_0$, which is not possible.   
\end{proof}

Indeed, one can show that $\diam Q^n$ converges exponentially fast to zero by showing that 
the changes of coordinates between different levels of renormalization are uniformly contracting. 
See Subsection 4.3 in \cite{Ch10-I}.  

\begin{proof}[Proof of Theorem~\ref{acc-on-cp}]
By the argument after Proposition~\ref{neighbor}, the sets 
\[E_n:=\{z\in J(P_\alf)\mid \orb(z) \text{ eventually stays in } \Omega^n\},\]
for $n\geq 0$, have full Lebesgue measure in $J(P_\alf)$.
Therefore, their intersection has full Lebesgue measure in $J(P_\alf)$.

Let $z\in \cap_{n=0}^\infty E_n\cap J(P_\alf)$. 
As $\cap_{j=0}^\infty \Omega^n\cap J(P_\alf)$ has zero area by Proposition~\ref{intersection-area} in the the Introduction and Theorem C 
in \cite{Ch10-I},  there are increasing sequences of positive integers $\{n_i\}_{i=1}^{\infty}$ and $\{t_i\}_{i=1}^\infty$ 
with the following property:
\[P_\alf\co{t_i}(z)\in \Omega^{n_i}\setminus \Omega^{n_i+1}, \text{ and }  P_\alf\co{t_i+1}(z)\in \Omega^{n_i+1}.\]
Now by Lemma~\ref{steps}, there are positive integers $s_i\geq t_i$ with $P_\alf\co{s_i}(z)\in Q_0^{n_i}$.
Combining with Lemma~\ref{nice&round}, we conclude that the orbit of $z$ gets arbitrarily close to the critical 
point of $P_\alf$. 
\end{proof}
\section{Perturbed Fatou coordinate}\label{S:perturbed-coordinate}
\subsection{The Pre-Fatou coordinate}
Let 
\[h(z)=\ea \cdot P \circ \vfi^{-1}(z)\colon \vfi(U)\ra \cc\] 
be a map in the class $\ff_\alf$ with $\alf\neq 0$, and let $\sigma_h$ denote the non-zero fixed point of $h$. 
For $\alf$ small enough, $h(z)$ has only two fixed points $0$ and $\sigma_h$ in a fixed neighborhood 
of $0$ (This property is guaranteed by assuming $\alf\in\irr$, \cite{IS06}). 
By Koebe distortion Theorem one can see that $\cp_h\in B(0,2)\setminus B(0,0.22)$. 

Consider the universal covering
\[\tau_h(w):= \frac{\sigma_h}{1-e^{-2\pi \Bi \alf w}}: \mathbb{C}\to \RS\setminus \{0,\sigma_h\}\] 
that satisfies $\tau_h(w+\alf^{-1})=\tau_h(w)$, for all $w \in \mathbb{C}$,  
$\tau_h (w)\ra 0$ as $\Im (\alf w) \ra\infty$, and $\tau_h (w) \ra \sigma_h$ as $\Im (\alf w) \ra -\infty$.

The map $h$ lifts under the covering $\tau_h$ to the map 
\[F_h(w):=w+\frac{1}{2\pi\alf \Bi}\log\big(1-\frac{\sigma_hu_h(z)}{1+zu_h(z))}\big),
\quad \text{with } z=\tau_\alf(w), u_h(z):= \frac{h(z)-z}{z(z-\sigma_h)}\]
defined on the set of points $w$ with $\tau_h(w)\in \Dom h$. 
The branch of $\log$ in the above formula is determined by $-\pi < \Im \log(\cdot) < \pi.$ 
We have  
\begin{equation}\label{lift-relation}
h \circ \tau_h=\tau_h \circ F_h, \quad \text{and}\quad F_h(w)+\alf^{-1}= F_h(w+\alf^{-1}).  
\end{equation}

The plan is to estimate the conformal mapping $L_h$ that conjugates $F_h$ to the translation by one.
For every real number $R>0$, let $\Theta(R)$ denote the set  
\[\Theta(R):=\cc \setminus \cup_{n \in \mathbb{Z}} B(n/\alf,R).\]
\begin{lem}\label{cyl-cond}
There exist positive constants $\eps_0$, $C_1$, $C_2$, and $C_3$ such 
that for every map $h \in \ff_\alf$ with $\alf\in (0,\eps_0)$ we have 
\medskip

\begin{enumerate}
\item[i.] The induced map $F_h$ is defined and is univalent on $\Theta (C_1)$. Moreover, on 
$\Theta (C_1)$
\begin{equation*}\label{I}
|F_h(w)-(w+1)|<1/4, \quad \text{and} \quad |F_h'(w)-1|<1/4.
\end{equation*} 
\item[ii.] For every $r\in (0,1/2)$ and every $w\in \Theta(r/\alf)\cap \Theta(C_1)$, we have  
\begin{equation*}\label{II}
|F_h(w)-(w+1)|< C_2 \frac{\alf}{r} e^{-2\pi\alf \Im w}.
\end{equation*}
\item[iii.] There exist a critical point $\cp_{F_h}\in B(0,C_1)$ of $F_h$, and the smallest non-negative integer 
$i(h)$ with 
\[F_h\co{i(h)}(\cp_{F_h})\in \Theta(C_1),\quad \text{and} \quad \sup_h i(h) <\infty.\]
\item[iv.] For every positive integer $j\leq i(h)+\frac{2}{3\alf}$, we have  
\[|F_h\co{j}(\cp_{F_h})-j|\leq C_3(1+\log j).\]  
\end{enumerate}
\end{lem}
\begin{proof}
{\em Part i. }Recall that the model polynomial $P(z)$ has only two critical points at $-1$ and $-1/3$, so it is univalent 
on the disk of radius $1/3$. 
The 1/4-Theorem guarantees that $\vfi_h(B(0,1/3))$ contains the disk of radius $1/12$. 
Hence, the composition $h(z)=\vfi_h^{-1}\circ P(\ea z)$ is univalent on the disk of radius $1/12$.

Note that 
\[\sigma_h=\frac{1-e^{2\pi\alpha \Bi}}{u_h(0)},\]
and 
\[\{u_h(0)\mid h\in \ff_{[0,\alpha^*]}\}\]
is compactly contained in $\cc\setminus \{0\}$.
Now by explicit estimates on $\tau_h$, there are constants $A_1$ and $\eps_0$ such that for all $h\in \ff_\alf$ 
with $\alf\in (0,\eps_0)$ and all $w\in \Theta(A_1)$, $|\tau_h(w)|\leq 1/12$.
By definition, $F_h$ must be univalent on $\Theta(A_1)$, for such maps $h$.

For the inequalities observe that 
\begin{align*}
|F_h(w)-(w+1)|&=|\frac{1}{2\pi \alpha\Bi}\log \big(1-\frac{\sigma_h u_h(z)}{1+z u_h(z)}\big)-1|\\
&= |\frac{1}{2\pi \alpha\Bi}\log \big(1-\frac{\sigma_h u_h(z)}{1+z u_h(z)}\big)- 
\frac{\log e^{2\pi \alpha\Bi}}{2\pi \alpha\Bi}|\\
&=\frac{1}{2\pi \alpha}|\log \big((1-\frac{\sigma_h u_h(z)}{1+z u_h(z)}) e^{-2\pi \alpha\Bi}\big)|.
\end{align*}
Also, replacing $\sigma_f$ by the above expression, 
\begin{align*}
|(1-\frac{\sigma_f u_f(z)}{1+z u_f(z)})e^{-2\pi\alpha\Bi}-1|
&= |(1-\frac{\sigma_f u_f(z)}{1+z u_f(z)})- e^{2\pi\alpha\Bi}|\\
& = |1-e^{2\pi\alpha\Bi}| |1-\frac{u_f(z)}{(1+zu_f(z))u_f(0)}|\\
&\leq 2\pi \alpha  |1-\frac{u_f(z)}{(1+zu_f(z))u_f(0)}|. 
\end{align*}
On the other hand, by the above argument, there exists a constant $A_1'$ (independent of $h$ and $\alpha$) 
such that the sets $\tau_h(\Theta(A_1'))$ are uniformly contained in a given neighborhood of zero.
Since the set of maps $u_h$ forms a compact set, the above estimates imply the first inequality. 
The second inequality on the set $\Theta(A_1'+1)$ follows from the first inequality and Cauchy Integral Formula. 
\medskip

{\em Part ii. }We further estimate the above expression as
\begin{align*}
 |1-\frac{u_h(z)}{(1+z u_h(z))u_h(0)}| 
&\leq 2|\frac{u_h'(0)-u_h^2(0)}{u_h(0)}| |z| \leq C_2' |\tau_h(w)|\leq C_2 \frac{\alf}{r} e^{-2\pi\alf \Im w},
\end{align*} 
for some constants $C_2', C_2$, by explicit estimates on the map $\tau_h$.
\medskip

{\em Part iii. } 
The unique critical point of $h$ at $\vfi_h(-1/3)$ lifts under $\tau_h$ into a $1/\alpha$-periodic set. 
By Part i, this set must contain an element in $B(0, C_2)$.
Indeed, by a simple estimate one can see that $U=\Dom \vfi_h$ contains the ball $B(0,2/3)$. 
Using the Koebe distortion Theorem, one can conclude that the critical point of $h$ belongs to the 
annulus $B(0,2)\setminus B(0,.22)$. 
This implies that the set of the critical points of $F_h$ in $B(0,C_2)$ is indeed contained in a compact subset of 
$\cc\setminus (\mathbb{Z}/\alpha)$. 
As $\alpha\to 0$, $h$ converges to a map in $\ff$, and for the maps in $\ff$ the orbit of the critical value is 
attracted by the parabolic fixed point at zero. 
By compactness of the class $\ff$, this implies that the chosen critical point of $ F_h$ leaves $B(0,C_2)$ 
in a uniformly bounded number of iterates.
\medskip

{\em Part iv. } 
This follows from the estimates in Part i, and Part ii. 
One may refer to \cite{Ch10-I} for a detailed proof.
\end{proof}
\subsection{The linearizing coordinate}
Let $h\in \ff_{(0,\alf_*]}$ with the perturbed Fatou coordinate $\Phi_h: \p_h\to \cc$.
The inverse image $\tau_h^{-1}(\p_h)$ has countably many components going from $-\Bi \infty$ to $+\Bi\infty$. 
Define 
\begin{equation}\label{linearize}
L_h:=\Phi_h\circ \tau_h: \Dom L_h \to \cc,
\end{equation}
where $\Dom L_h$ is, by definition, the component of $\tau_h^{-1}(\p_h)$ separating $0$ and $1/\alf$ in $\cc$.
It follows from Theorem~\ref{Ino-Shi1} that the map $L_h: \Dom L_h\to \cc$ is univalent, 
$L_h(\cv_{F_h})=1$, $L_h(\Dom L_h)$ contains the set  
\[\{w\in\cc : 0< \Re(w) < \lfloor 1/\alf\rfloor-\Bk \},\]
$L_h(w)\to +\Bi\infty$ as $\Im (w)\to +\infty$, and  $L_h(w)\to -\Bi\infty$ as $\Im (w)\to -\infty$.
Moreover, $L_h$ satisfies the Abel functional equation   
\begin{equation}\label{equivariance}
L_h(F_h(w))=L_h(w)+1,
\end{equation} 
whenever both sides are defined.

Using~\eqref{equivariance}, if $\alf< \eps_0$, $L_h$ extends to a univalent map on the set 
\begin{multline*}
\Sigma_{C_1}:=\{w\in\cc : C_1 \leq \Re(w)\leq \alf^{-1}-C_1\} \\ 
\cup \{w \in \cc : \Re(w)\geq \alf^{-1}-C_1, \text{ and } |\!\Im w|\geq \Re(w)-\alf^{-1}+2C_1\}\\
\cup\{w \in \cc : \Re(w)\leq C_1, \text{ and } |\!\Im w|\geq -\Re(w)+2C_1\}.
\end{multline*}
More precisely, by Lemma~\ref{cyl-cond}-i, for every $w\in \Sigma_{C_1}$ there is an integer $j$ with 
$F_h^j(w)\in \Dom L_h$. 
One defines $L_h (w):= L_h(F_h^j(w))-j$. 
We denote the extended map by the same notation $L_h$.

\begin{lem}\label{derivative}\label{bounded-derivative}
There exists a positive constant $C_4$ such that for every $h\in \ff_\alf$, and every 
$\zeta\in L_h(\Dom L_h)\setminus B(0,1/2)$ we have $1/C_4 \leq |{(L_h^{-1}})'(\zeta)|\leq C_4$. 
Moreover, $L_h'(z)\to 1$, as $\Im z\to \infty$.
\end{lem}    
\begin{proof}
This follows from the the Koebe distortion Theorem, and the compactness of the class $\ff$.
\end{proof}
\subsection{The model} 
Given $A\in \mathbb{C}$, let $\K_h$ denoted the domain (with asymptotic width one) bounded by the curves 
$A+t\Bi$, $F_h(A+t\Bi)$, $(1-s)A+sF_h(A)$, for $t\in [0,\infty)$ and $s\in [0,1]$.
Moreover, assume that $\K_h$ is contained in $\Theta(C_1+1)\cap \Theta(r/\alf)$. 

For $(s,t) \in [0,1]\times [0,\infty) $ define the map 
\[H(s,t):=A+ \int_{0}^s X(\ell,t) \, d\ell+\Bi \big( t+ \int_{0}^s Y(\ell,t) \, d\ell\big ) ,\]
where $X$ and $Y$ are given as
\begin{gather*}  
X(s,t):=a_0(t)+a_1(t) \sin(\pi s)+a_2(t) \cos(\pi s) +a_3(t) \sin (2\pi s)+ a_4(t) \cos (2\pi s), \\
Y(s,t):=b_0(t)+b_1(t) \sin(\pi s)+b_2(t) \cos(\pi s) +b_3(t) \sin (2\pi s)+b_4(t) \cos (2\pi s),
\end{gather*}
with
\begin{align*}
a_0(t)&=\Re (F_h(A+\Bi t)-A)+\Re F_h''(A+\Bi t)/\pi,&&\\
b_0(t)&=\Im(F_h(A+\Bi t)-A)-t+\Im (F_h''(A+\Bi t))/\pi,\!\!\!\!&&\\
a_1(t)&=-\Re F_h''(A+\Bi t)/2, &&   b_1(t)=-\Im F_h''(A+\Bi t)/2, \\
a_2(t)&=(1-\Re F_h'(A+\Bi t))/2, && b_2(t)=-\Im F_h'(A+\Bi t)/2,\\ 
a_3(t)&=\Re F_h''(A+\Bi t)/4, &&  b_3(t)=\Im F_h''(A+\Bi t)/4, \\
a_4(t)&=1-a_0(t)-a_2(t), &&  b_4(t)=-a_0(t)-a_2(t).\\    
\end{align*}
(The above coefficients are obtained from solving a separable partial differential 
equation so that the next lemma holds.)
 
 Instead of algebraically manipulating constants and frequently introducing new ones, we use the 
following convention from now on. 
Given real valued functions $f$ and $g$ defined on a set $W$, we write 
\[ f(x)\preceq g(x) \text{  on } W \] 
to mean that there exists a constant $M$ with $f(x)\leq M g(x)$ for all $x\in W$. 
For example, by Lemma~\ref{cyl-cond} and Cauchy Integral Estimate, on the set 
$\Theta(r/\alf)\cap \Theta(C_1+1)$ we have  
\begin{equation}\label{F-derivatives}
\max \big\{ |F_h(w)-w-1|, |F_h'(w)-1|,|F_h''(w)|,|F_h^{(3)}(w)|,|F_h^{(4)}(w)|\big\}
\preceq\frac{\alf}{r} e^{-2\pi\alf\Im w}.
\end{equation}

Using the relation $H(s+1,t)=F_h(H(s,t))$ we extend $H$ onto a neighborhood of the form
$(-\delta,1+\delta)\times (0,\infty)$, for some $\delta>0$.

\begin{lem}\label{H-properties}  
The map $H$ satisfies the following properties:
\begin{itemize}
\item[i.] For every $t\in (0,\infty)$ and $s\in (-\delta, \delta)$ we have \[F(H(s,t))=H(s+1,t).\]
\item[ii.] $H$ is $C^2$ on $(-\delta,1+\delta)\times (0,\infty)$. 
\item[iii.] On $(-\delta,1+\delta)\times (0, \infty)$, we have
\[\max\big\{|\partial_s H(s,t)-1|, |\partial_t H(s,t)-\Bi|\big\} \preceq\frac{\alf}{r} e^{-2\pi \alf (t+\Im A)},\]
and
\[\max\big\{|\partial_{ss} H(s,t)|, |\partial_{tt} H(s,t)|, |\partial_{st} H(s,t)|\big\}
 \preceq\frac{\alf}{r} e^{-2\pi \alf (t+\Im A)}.\]
\end{itemize} 
\end{lem}
\begin{proof}
Clearly $H(0,t)=A+t\Bi$. 
On the other hand
\begin{align*}
H(1,t)&=A+a_0(t)+2a_1(t)/\pi + (t+b_0(t)+2b_1(t)/\pi)\Bi\\
&=F(A+t\Bi)
\end{align*}
For other values of $s$, the relation follows from the definition.

The map $H$ is real analytic in the domain  
$(-\delta,1+\delta)\times(0,\infty)\setminus\{0,1\}\times(0,\infty)$.
A straightforward calculation shows that it is $C^2$ on the lines $\{0,1\}\times(0,\infty)$.

From Inequality \eqref{F-derivatives}, on $[0,\infty)$ we have 
\begin{gather*}
|a_0(t)-1|\preceq \frac{\alf}{r} e^{-2\pi \alf (t+\Im A)},\\
|a_j(t)| \preceq \frac{\alf}{r} e^{-2\pi \alf (t+\Im A)}, \text{ for } j=1,2,3,4, \\
|b_j(t)| \preceq \frac{\alf}{r} e^{-2\pi \alf (t+\Im A)}, \text{ for } j=0,1,2,3,4.
\end{gather*}
Similar estimates hold for the first and second derivatives of these functions. 
This implies that on $[0,1]\times (0, \infty)$
\[\max \{ |X-1|, |\partial_s X|, |\partial_t X|, |\partial_{tt}X|, |Y|,|\partial_t Y|,|\partial_{tt} Y|\} (s,t)
\preceq \frac{\alf}{r} e^{-2\pi \alf (t+\Im A)}.\]
One infers the inequalities in Part iii from these estimates.
\end{proof}

To simplify the calculations, from now on we use the complex notation $z=s+\Bi t$, $dz=ds+\Bi dt$, 
$d\bar{z}=ds-\Bi dt$, $\partial_z=(\partial_s-\Bi \partial_t)/2$, 
$\partial_{\bar{z}}=(\partial_s+\Bi\partial_t)/2$.  

To compare $L_h^{-1}$ to $H$ we work on the map 
\[G:= L_h\circ H:(-\delta,1+\delta)+ \Bi(0,\infty) \to \cc.\] 
In particular, we would like to prove that 
\begin{equation}\label{G-zz}
|\partial_{z} G(z)-1| \preceq \frac{1}{r} e^{-2\pi \alf \Im z}.
\end{equation}

\begin{lem}\label{G-properties}
The map $G$ has a $C^2$ extension onto $\cc$ which satisfies 
\begin{equation}\label{E:periodic}
G(z+1)=G(z)+1,\quad  \forall z\in \cc,
\end{equation}
and is the translation by $G(0,0)$ on $\{z\mid \Im z\leq -1\}$.
\begin{itemize}
\item[i.] On $\{z\mid \Im z\geq 0\}$,  
\[\max\{|\partial_{\bar{z}}G(z)|, |\partial_{z\bar{z}} G(z)|\}
\preceq  \frac{\alf}{r} e^{-2\pi \alf \Im (z+A)},\]
and on $\{z\mid \Im z \in [-1,0]\}$,  
\[\max\{|\partial_{\bar{z}}G(z)|, |\partial_{z\bar{z}} G(z)|\}\preceq 1.\]
\item[ii.] There exists a constant $z_0$ such that  as $\Im z\to \infty$, $G(z)-z \to z_0$.
\item[iii] For every constant $\delta\in \mathbb{R}$ there is a constant $C_7$ independent of $r$ such that if $\K_h\subseteq \Theta(r/\alf)$ and $\liminf_{w\in \K_h} \Im w\geq \delta/\alf$ then for every 
$z_1,z_2\in \Dom G$,
\[|G(z_2)-G(z_1)-(z_2-z_1)|\leq C_7/r.\]
\end{itemize}
\end{lem}
\begin{proof}
From Lemma~\ref{H-properties}-i and functional equation \eqref{equivariance}, 
the Relation~\eqref{E:periodic} holds for every $z$ with $\Re z\in (-\delta,\delta)$ and $\Im z>0$. 
Moreover, it is $C^2$ as $H$ is $C^2$ smooth and $L_h$ is analytic. 
We can use \eqref{E:periodic} to extend $G$ onto the upper half plane.
 
Below the line $\Im z=-1$ we define $G(z):=z+G(0,0)$. 
Then, there exists a $C^2$ smooth interpolation of $G$ on the strip $\Im z\in [-1,0]$  that satisfies \eqref{E:periodic}, 
and has uniformly bounded first and second order partial derivatives.    
(For example, one may obtain this interpolation by first $C^1$ gluing the vector field $\partial G/\partial s$ 
over the upper half plane to the constant vector field 1 below the line $\Im z=-1$, and then integrate this 
vector field.)
\medskip

{\em Part i:}
By complex chain rule, with $w=H(z)$, we have 
\begin{align*}
|\partial_{\bar{z}} G(z)|&=|\partial_w L_h (H(z))\cdot \partial_{\bar{z}} H(z)|\\
&\preceq \frac{\alf}{r} e^{-2\pi \alf \Im (z+A)},
\end{align*}
by Lemmas~\ref{derivative} and \ref{H-properties}-iii.

On the other hand, Lemma~\ref{derivative} and the Cauchy Integral Formula imply that $|\partial _{ww} L_h|$ 
is uniformly bounded well inside the domain of $L_h$. 
Therefore, from the estimates in Lemma~\ref{H-properties}, and the complex chain rule
\begin{align*}
|\partial_{z\bar{z}} G(z)|&=|\partial_{ww}L_h(H(z))\cdot \partial_z H(z)\cdot \partial_{\bar{z}} H(z)+ 
\partial_w L_h( H(z))\cdot \partial_{z\bar{z}} H(z)|\\
&\preceq \frac{\alf}{r} e^{-2\pi \alf \Im (z+A)}.
\end{align*}
\medskip

{\em Part ii:}
Since $G$ is periodic of period one, it is enough to prove the asymptotic behavior at $z$ with 
$\Re z\in [0,1]$. 
In this region, it is enough to prove that for every $\eps>0$ there is $R\in \mathbb{R}$ such that if 
$\Im z_1$ and $\Im z_2$ are larger than $R$ then $|(G(z_2)-z_2)-(G(z_1)-z_1)|$ is less than $\eps$. 

Given real constants $t_2>t_1$, define 
$\D:=\{z\in \mathbb{C}\mid 0\leq \Re z \leq 1,  t_1\leq \Im z \leq t_2\}$. 
By Green's Integral Formula, we have 
\begin{align} \label{Green}
\int_{\partial \mathcal{D}} G(z)\, dz=\iint_{\mathcal{D}} \partial_{\bar{z}} G(z) \,d\bar{z} dz.
\end{align}

The right-hand integral is bounded by 
\begin{align*}
\Big|\iint_\mathcal{D} \partial_{\bar{z}} G(z)  \, d z d\bar{z}\Big| &
\leq 2\int_{t_1}^{t_2}\!\!\! \int_0^1 |\partial_{\bar{z}} G(s,t)|\, ds dt  &   \\
&\preceq\int_{t_1}^{t_2}\!\frac{\alf}{r} e^{-2\pi\alf (t+\Im A)}\,ds dt &\text{(Lemma~\ref{G-properties}-ii)} \\
& \leq \frac{e^{-2\pi \alf \Im A}}{2\pi r} e^{-2\pi \alf t_1}.
\end{align*}
This bound tends to zero, as $t_1$ goes to infinity.

By periodicity of $G$, the left hand side of \eqref{Green} is equal to  
\begin{eqnarray*}
-\Bi(t_2-t_1)+\int_{\gamma_1} G(z)\,d s + \int_{\gamma_2} G(z)\,d s,
\end{eqnarray*}
where $\gamma_1$ and $\gamma_2$ are the top and bottom boundaries of $\D$ with appropriate 
orientations. 
By Lemmas~\ref{derivative} and \ref{H-properties}-iii, as $\Im z\to \infty$, $\partial_z G(z)\to 1$.
Hence 
\[ \int_{\gamma_1} G(z)\,d s\to -\int_0^1 G(t_1\Bi)+1\cdot s \,d s=-G(t_1\Bi)-1/2\]
and 
\[ \int_{\gamma_2} G(z)\,d s\to \int_0^1 G(t_2\Bi)+1\cdot s \,d s= G(t_2\Bi)+1/2.\]

Putting these together, we conclude that as 
$\Im z_1, \Im z_2 \to \infty$, 
\begin{multline*}
|(G(z_2)-z_2)- (G(z_1)-z_1)| \\  \to  |(G(\Bi \Im z_2)-\Bi\Im z_2)- (G(\Bi\Im z_1)-\Bi\Im z_1)| \to 0.
\end{multline*}
This finishes the proof of this part.
\medskip

{\em Part iii:}
This follows from the above relations and inequalities once we have a lower bound on $Im z_1$ and $\Im z_2$. 
(See \cite{Ch10-I}, proof of Lemma 5.4 for further details.)
\end{proof}

The map $G$ projects via $e^{2\pi\Bi z}$ to a well-defined map on $\mathbb{C}$.  
More precisely, let 
\[\phi(z):= e^{2\pi\Bi z}, \psi(\zeta):= \phi^{-1}(\zeta)=\frac{1}{2\pi \Bi}\log \zeta, \text{ with }
\Im \log(\cdot)\in [0,2\pi), \] 
and define   
\[K(\zeta):=\phi \circ  G \circ \psi(\zeta).\]
The map $K$ has  continuous extension to $0$; $K(0)=0$.

\begin{lem}\label{K-properties}
The map $K:\mathbb{C}\to \mathbb{C}$ is complex differentiable at $0$,  and 
\[|\partial_\zeta K(\zeta)-\partial_\zeta K(0)|\preceq \frac{1}{r} |\zeta|^\alf.\]
\end{lem}
\begin{proof}
By the definition of derivative and Lemma~\ref{G-properties}-iii,
\begin{equation}\label{K'(0)}
\begin{aligned}
\partial_\zeta K(0)&= \lim_{\zeta\to 0} \frac{e^{2\pi\Bi G(\frac{1}{2\pi \Bi}\log \zeta)}}{\zeta}\\
&=\lim_{\zeta\to 0} e^{2\pi\Bi (G(\frac{1}{2\pi \Bi}\log \zeta)-\frac{1}{2\pi \Bi}\log \zeta)}\\
&=e^{2\pi\Bi z_0}.
\end{aligned}
\end{equation}
\medskip

To prove the inequality, first we estimate the Laplacian of $K$ on $\mathbb{C}\setminus \{0\}$. 
Using the complex chain rule, with $\zeta=\phi(z)$ and $\psi(\zeta)=z$, at $\zeta\neq 0$
\begin{equation}\label{K'(zeta)}
\begin{aligned}
\partial_\zeta K=\partial_\zeta (\phi\circ G\circ\psi) &= \partial_z \phi \circ (G\circ \psi)\cdot \partial_\zeta (G\circ \psi) +
\partial_{\bar{z}} \phi \circ (G\circ \psi)\cdot \partial_\zeta (\overline{G\circ \psi}) \\
&=\partial_z \phi \circ G\circ \psi \cdot \partial_\zeta (G\circ \psi)\\
&=\partial_z \phi \circ G\circ \psi\cdot \partial_z G \circ \psi \cdot \partial_\zeta \psi, 
\end{aligned}
\end{equation}
therefore,  
\begin{align*}
\partial_{\bar{\zeta}}\partial_\zeta K& =\partial_{\bar{\zeta}} \left(\partial_z \phi \circ (G\circ \psi)\right)
\cdot \partial_z G \circ \psi \cdot \partial_\zeta \psi  \\
 &\qquad +\partial_z \phi \circ (G\circ \psi)\cdot \partial_{\bar{\zeta}}(\partial_z G \circ \psi)\cdot \partial_\zeta\psi\\
&= \left( \partial_{zz} \phi \circ G\circ \psi \cdot \partial_{\bar{z}} G \circ \psi\cdot \partial_{\bar{\zeta}} \bar{\psi}\cdot  \partial_{z} G \circ \psi  \right) \cdot \partial_z \psi \\
& \qquad+ \partial_{z} \phi \circ G \circ \psi \cdot 
\left((\partial_{zz}G\circ \psi)\cdot \partial_{\bar{\zeta}}\psi + \partial_{z\bar{z}} G\circ \psi \cdot \partial_{\bar{\zeta}}\bar{\psi} \right) 
\cdot \partial_z \psi \\
&=\partial_{zz} \phi \circ G\circ \psi \cdot \partial_{\bar{z}} G \circ \psi\cdot \overline{\partial_{\zeta} \psi}\cdot  \partial_{z} G \circ \psi \cdot \partial_z \psi \\
& \qquad+ \partial_{z} \phi \circ G \circ \psi \cdot 
\partial_{z\bar{z}} G\circ \psi \cdot \overline{\partial_{\zeta} \psi} \cdot \partial_z \psi \\
&= 2\pi \Bi K |\partial_\zeta \psi|^2\left(\partial_{z\bar{z}} G\circ\psi + 2\pi\Bi \partial_{\bar{z}} G
\circ \psi \cdot \partial_z G\circ \psi\right).
\end{align*}
(To get the above expression, one could also find the coefficient of $(\zeta-\zeta_0)(\bar{\zeta}-\bar{\zeta_0})$ in the expansion of $K(\zeta)$ near $\zeta_0$, as $K$ is real analytic away from the positive real axis.) 

Since $G$ is analytic below the horizontal line $\Im z=-1$, $\partial_{\bar{\zeta}\zeta}K(\zeta)=0$ outside of the disk of radius $e^{2\pi}$. 
Above this line, we have the estimates in Lemma~\ref{G-properties}-i that provide us with  
\[|\partial_{\zeta\bar{\zeta}}K(\zeta)|\preceq \frac{\alf}{r} |\zeta|^{\alf-1}.\]

Fix $\zeta_0\in \mathbb{C}$ and choose a disk $B(0, R)$ of radius $R>e^{2\pi}$ containing $\zeta_0$. 
The general form of Cauchy Integral Formula states that for the continuous function $\partial_{\zeta} K$
\begin{align*}
\partial_{\zeta}K(\zeta_0)&= \frac{1}{2\pi \Bi} \int_{\partial B(0,R)} \frac{\partial_\zeta K (\zeta)}{\zeta-\zeta_0} d\zeta
+\frac{1}{2\pi\Bi} \iint_{B(0,R)}\frac{\partial_{\bar{\zeta}} \partial_\zeta K(\zeta)}{\zeta-\zeta_0} d\zeta d\bar{\zeta} 
\end{align*}
Outside of the disk $B(0,e^{2\pi})$, $\partial_\zeta K=e^{2\pi \Bi G(0,0)}$ and $\partial{\zeta\bar{\zeta}}=0$ by 
Lemma~\ref{G-properties}, thus the above formula reduces to 
\begin{align*}
\partial_{\zeta}K(\zeta_0)=1+ \frac{1}{2\pi\Bi} 
\iint_{B(0,e^{2\pi})} \frac{\partial_{\bar{\zeta}\zeta} K(\zeta)}{\zeta-\zeta_0} d\zeta d\bar{\zeta}. 
\end{align*}
 
We may now estimate the difference at $\zeta_0\in B(0,1)$ as 
\begin{align*}
|\partial_{\zeta}K(\zeta_0)-\partial_{\zeta}K(0)|&\leq
\frac{1}{2\pi}\iint_{B(0,e^{2\pi})} 
|\partial_{\bar{\zeta}\zeta} K(\zeta)|\frac{|\zeta_0|}{|\zeta-\zeta_0|\cdot|\zeta|} d\zeta d\bar{\zeta} \\
&\preceq \frac{\alf}{r} |\zeta_0|^\alf 
\iint_{B(0,e^{2\pi})} \frac{|\zeta_0|^{1-\alf}}{|\zeta-\zeta_0|\cdot|\zeta|^{2-\alf}} d\zeta d\bar{\zeta} \\
&\preceq \frac{1}{r} |\zeta_0|^\alf.
\end{align*}
The last inequality is obtained by virtue of the following calculations. 

Define $B_1:=B(0,|\zeta_0|/2)$, $B_2:=B(\zeta_0,|\zeta_0|/2)$, and $B_3:=B(0,e^{2\pi})\setminus(B_1\cup B_2)$. Note that on $B_1$, we have $|\zeta-\zeta_0|\geq |\zeta_0|/2$, on $B_2$ we have $|\zeta|\geq |\zeta_0|/2$, and on $B_3$ we have $|\zeta-\zeta_0|\geq |\zeta|/2$. 
Hence 
\begin{align*}
\iint_{B_1} \frac{|\zeta_0|^{1-\alf}}{|\zeta-\zeta_0|\cdot|\zeta|^{2-\alf}} d\zeta d\bar{\zeta}
&\leq \frac{2}{|\zeta_0|^\alf} \iint_{B_1} \frac{1}{|\zeta|^{2-\alf}} d\zeta d\bar{\zeta} 
= \frac{2^{3-\alf}\pi}{\alf}, \\
\iint_{B_2} \frac{|\zeta_0|^{1-\alf}}{|\zeta-\zeta_0|\cdot|\zeta|^{2-\alf}} d\zeta d\bar{\zeta}
&\leq \frac{2^{2-\alf}}{|\zeta_0|} \iint_{B_2} \frac{1}{|\zeta-\zeta_0|} d\zeta d\bar{\zeta}  
=2^{3-\alf} \pi,
\end{align*}
and 
\begin{align*}
\iint_{B_3} \frac{|\zeta_0|^{1-\alf}}{|\zeta-\zeta_0|\cdot|\zeta|^{2-\alf}} d\zeta d\bar{\zeta}
&\leq 2|\zeta_0|^{1-\alf} \iint_{B_3} \frac{1}{|\zeta|^{3-\alf}} d\zeta d\bar{\zeta}\\
&\leq  2|\zeta_0|^{1-\alf} \iint_{B(0,e^{2\pi}) \setminus B_1} \frac{1}{|\zeta|^{3-\alf}} d\zeta d\bar{\zeta}\\
&= \frac{8\pi |\zeta_0|^{1-\alf}}{\alf-1}(e^{2\pi(\alf-1)}-(|\zeta_0|/2)^{\alf-1})\\
&\leq \frac{2^{4-\alf}\pi}{1-\alf}.
\end{align*}
\end{proof}

\begin{lem}
We have Inequality \eqref{G-zz} on the upper half plane. 
\end{lem}
\begin{proof}
This results from Lemmas~\ref{K-properties} and \ref{G-properties}-ii, using Equations \eqref{K'(0)} 
and~\eqref{K'(zeta)}.
\end{proof}

\begin{rem}
An alternative approach to get Inequality~\eqref{G-zz} from the model map $H$ is using the Beltrami 
differential equation. 
It follows from the properties of $H$ that the complex dilatation of $H$, $\mu:=\partial_{\bar{z}} H/\partial z H$,
can be extended to a $C^{1+\epsilon}$ map onto the upper half plane, using the relation $\mu(z+1)=\mu(z)$. 
The function $\mu$ has absolute value strictly less than $1$ at points with large imaginary part. 
The Beltrami equation $\mu \partial_z G=\partial_{\bar{z}} G$ has a $C^2$ periodic solution with Fourier 
expansion. 
Indeed, one can find the coefficients of this expansion in terms of $a_j$'s and $b_j$'s, by comparing the 
coefficients in the Beltrami equation term by term. 
\end{rem}
\subsection{Main Estimate}
\begin{propo}[main estimate]\label{fine-tune} 
There exists a constant $M$ such that for every $r\in (0,1/2]$, and every 
$w\in \Dom L_h\cap\Theta(r/\alf)\cap\Theta(C_1+1)$, we have 
\begin{equation}\label{E:main-estimate}
|L_h'(w)-1|\leq \frac{M}{r} e^{-2\pi \alf \Im w}. 
\end{equation}
\end{propo}
\begin{proof} 
Since $|L_h'(w)|$ is uniformly bounded by Lemma~\ref{derivative}, it is enough to prove the proposition for 
$w$ with $\Im w\geq 0$. 
Given such $w$, we choose a $\K_h \subset \Theta(r/\alf) \cap\Theta(C_1+1)$ containing $w$, and 
consider the corresponding map $H$ and $G$. (Here $\Im A\geq 0$.) 

By Lemmas \ref{H-properties} and \ref{G-properties}, at $z:=H^{-1}(w)$ we have 
\[|\partial_z H(z)-1| \leq M_1 \frac{\alf}{r} e^{-2\pi\alf \Im w}, \text{ and }   
|\partial_z G(z)-1| \leq M_2 \frac{1}{r} e^{-2\pi\alf \Im w},\]
for some constants $M_1$ and $M_2$ independent of $r$ and $w$ (and $G$, $H$).

Applying $\partial_z$ to $G= L_h\circ H$ at $w=H(z)$, we obtain 
\[\partial_{z} G(z)=\partial_{w} L_h (w)\cdot \partial _z H(z).\]
This implies the desired inequality.
\end{proof}
\subsection{Proof of the main technical lemmas}
\begin{lem}\label{formula-derivative}
For every $h\in \fff$ and $z\in \mathbb{C}$, let $\mathbb{L}og$ denote an arbitrary inverse branch of $\ex$ 
containing $\tau_h(z)$ in the interior of its domain. 
We have 
\[|(\mathbb{L}og\circ \tau_h)'(z)-\alf|\leq C_8\alf e^{-2\pi \alf \Im z},\]
where $C_8$ is a constant independent of $h$ and $w$.
\end{lem}
\begin{proof}
The proof follows from explicit calculations. 
\end{proof}

\begin{proof}[Proof of Proposition~\ref{main-estimate}]
By definition, $\chi_h:= \mathbb{L}og \circ (\Phi_h^{\dagger})^{-1}$, and 
$(\Phi_h^{\dagger})^{-1}=\tau_h\circ L_h^{-1}.$ 
To estimate the derivative of $\chi_h$, first we estimate the derivative of the inverse map $L_h^{-1}$ 
at $w$ using Lemma~\ref{fine-tune}.
To this end, we need to locate $L_h^{-1}(w)$.

Given $w\in \Theta(r,\alf)\cap \Dom \chi_h$, by pre-compactness of the class $\fff$, and Lemma~\ref{derivative},
there exists a $r'$, with $r'/r$ uniformly bounded, such that $L_h^{-1}(w)$ belongs to $\Theta(r'/\alf)$. 

On the other hand, from Lemma~\ref{cyl-cond}-iv, we know that $\Im L_h^{-1}(1/2\alf)$ is bigger than 
$-C_3\log (1+1/2\alf)$.
Now, from the first part of the same lemma and the relation $L_h(F_h(w))=L_h(w)+1$ one concludes
that  $\Im L_h^{-1}(w)$ must be at least 
\[-C_3\log (1+1/2\alf)-1/4\alf.\]
Now choose $\delta\in \mathbb{R}$ such that
\[-C_3\log (1+1/2\alf)-1/4\alf\geq\delta/\alf\]
holds for every $\alf\in (0,1)$.
The above argument implies that 
\begin{align*}
\Im L_h^{-1}(w)&\geq -C_3\log (1+1/2\alf)-1/4\alf+\Im L_h^{-1}(\Bi\Im w+1/2\alf),
\end{align*}
which is at least $\delta/\alf + \Im w-4(C_7+8)$, by the second statement in Lemma~\ref{G-properties}-iii, and 
a similar property for $H$.

Now we may use Lemma~\ref{fine-tune}, at $w':=L_h^{-1}(w)$, to obtain 
\[|L_h'(w')-1|\leq \frac{M'}{r} e^{-2\pi \alf \Im w},\]
for an appropriate constant $M'$.

With $z=L_h^{-1}(w)$, 
\begin{align*}
|\chi'_h(w)-\alf| &\leq |(\mathbb{L}og\circ \tau_h)'(z)\cdot (L_h^{-1})'(z)-
(\mathbb{L}og\circ \tau_h)'(w)+(\mathbb{L}og\circ \tau_h)'(z)-\alf| \\
&\leq |(\mathbb{L}og\circ \tau_h)'(z)| |(L_h^{-1})'(w)-1|+|(\mathbb{L}og\circ \tau_h)'(z)-\alf| \\
&\leq C\frac{\alf}{r} e^{-2\pi \alf \Im w},
\end{align*}
for some constant $C$.
\end{proof}
\subsection*{Acknowledgment} 
I would like to thank X. Buff and A. Ch\'eritat for several helpful discussions around the proof of 
Theorem~\ref{pc-area}.
Also, I acknowledge the financial support from Leverhulme Trust in London and Institut de Math\'{e}matiques 
de Toulouse while carrying out this research.
\bibliographystyle{smfalpha}
\bibliography{Main.bbl}
\end{document}